\RequirePackage{fix-cm}
\documentclass[smallcondensed]{svjour3}     
\usepackage{amssymb}
\usepackage{amsmath}
\usepackage{dsfont}
\usepackage{bbold}
\usepackage[dvips]{graphics}
\usepackage{anyfontsize}
\usepackage{xcolor}
\usepackage{color}
\usepackage[english]{babel}
\usepackage{pstricks-add}
\usepackage{graphicx}
\usepackage{pgf,tikz}
\usepackage{mathrsfs}
\usepackage{mathtools}
\usepackage{ulem}
\usepackage{float}
\usepackage{pagecolor}
\usepackage{hyperref}
\usepackage{booktabs}
\usepackage{orcidlink}
\usetikzlibrary{matrix,arrows}
\usepackage{tikz-cd}
\usepackage{multicol}
\usepackage{amssymb,relsize}
\newcommand{\mathfrakO}{\mathpalette\bigmathfrakO\relax}
\newcommand{\bigmathfrakO}[2]{\scalebox{1.2}{$#1\mathfrak{0}$}}
\newcommand{\mathfrakC}{\mathpalette\bigmathfrakC\relax}
\newcommand{\bigmathfrakC}[2]{\scalebox{1.2}{$#1\mathfrak{c}$}}
\smartqed  
\usepackage{graphicx}

\newtheorem{dfn}{Definition}
\numberwithin{equation}{section}
\begin{document}

\title{Decomposition of Symmetrical Classes of Central Configurations}
%



\author{Marcelo P. Santos\orcidlink{0000-0002-9023-0728}       \and
        Leon D. da Silva\orcidlink{0000-0002-2823-3418}} 


\institute{Federal Rural University of Pernambuco \\
    Department of Mathematics\\
    Dom Manoel de Medeiros Street, S/N, Dois Irmãos\\
    Recife, PE\\
    52171-900, Brazil\\
    \email{marcelo.pedrosantos@ufrpe.br} \\
    \email{leon.silva@ufrpe.br}
    }          

\date{Received: date / Accepted: date}

\maketitle

\begin{abstract}
We study central configurations when the set of positions is symmetric.  We use a theorem(proved in \cite{SANTOS_22}) that allows us to use the representation theory of finite groups to explore the symmetry properties of equations for central configurations. This approach simplifies equations for central configurations by considering arbitrary numbers of bodies, symmetry groups, and dimensions.
We discuss how to use this theorem to obtain a more refined decomposition of the equations than that given before. The decomposition presented here uses the symmetry-adapted basis method.

As an application, we give a complete description of the existence and which masses are possible for central configurations of two nested regular
tetrahedrons, two nested regular octahedrons, and two nested regular cubes.
To do this, we employ some methods of rational parameterizations and isolation of zeros of multivariate polynomials. The decomposition obtained allows symbolic calculations to be used to study the expressions.
This way, we summarized the same discussion of works done in \cite{Corbera_Llibre_2n,ZHU,Liu_Tao} and extended them by completing the discussion on the 
cube case, in the inverse and direct problems.
\end{abstract}

\section{Introduction}

\label{sec:1}
The N-body problem describes the dynamics of point masses under the action of the gravitational law of attraction. Let $m_i$  represent a positive point mass at position $q_i \in \mathbb{R}^d$. By Newton's second law, the equations of motion are
\begin{equation}
\label{NbodyProblem}
m_i\ddot q_i=\sum_{j \neq i} m_im_j\frac{q_j-q_i}{|q_j-q_i|^3},\quad i=1,\ldots, N.
\end{equation}
These equations are hard enough to prevent finding all solutions if $N>2$.
 Nevertheless, it is possible to derive particular simple solutions. 
 Consider the {\it center of mass} given by
\begin{equation}
     q_0=\frac{m_1q_1+\cdots+m_Nq_N}{m_1+\cdots+m_N}.\label{CentroDeMassa}
 \end{equation}
Suppose that each body has the acceleration vector pointing towards the center of mass with the magnitude proportional to the distance to the center of mass. So there exists some positive $\lambda \in \mathbb{R}$ such that
\begin{equation}
\ddot q_i=-\lambda(q_i-q_0) \quad \forall i=1,\ldots, N.\label{AceleracaoPraOCentroDeMassa}
\end{equation}
Then, releasing the bodies without initial velocities, the system collapses homothetically on its center of mass. These are the so-called {\it homothetic solutions}. Using \eqref{NbodyProblem} and \eqref{AceleracaoPraOCentroDeMassa}, we conclude that at any given time, a homothetic solution satisfies the equation
  \begin{equation}
-\lambda(q_i- q_0)=\sum_{j \neq i} m_j\frac{q_j-q_i}{|q_j-q_i|^3}, \quad i=1,\ldots, N.\label{EquacaoDeConfiguracaoCentral}
\end{equation} 
More generally, if a configuration $q=(q_1,\ldots,q_N)\in\mathbb{R}^{dN}$
 satisfies equation \eqref{EquacaoDeConfiguracaoCentral}
 for some $\lambda \in \mathbb{R}$, we say that $q$
 is a {\it central configuration}. 
 
 Although it is possible a collapse of the system not being homothetic, in any collapse, the limiting configuration tends to be a central configuration (see \cite{saari}).
 
The central configurations give rise to self-similar solutions for the N-body problem. Self-similar means that given two times, the solutions differ only by rotation, translation, or dilatation. If we restrict the dimension $d\leq 3$,  a self-similar solution is possible only if the initial condition is a central configuration.
  If the dimension $d=2$, these solutions are called {\it relative equilibrium}, meaning they are an equilibrium solution in a rotating frame of reference if the convenient angular velocity is chosen.
 
  Central configurations also play an important role in determining the topology of the integral manifolds of the N-body problem
  (see, for example, \cite{SmaleIntegralManifolds,CabralIntegralManifolds,AlbouyIntegralManifolds}).
 For good introductions on central configurations, see \cite{MeyerHallOffin,LlibreMoeckelSimo}. 
 
 From the equations \eqref{EquacaoDeConfiguracaoCentral}, we see that given a central configuration, we obtain new ones by rotation, translation, reflection, and dilatation (centered at the center of mass). So the central configurations are counted up to these transformations.
 
 Among the most important problems in Celestial Mechanics is the problem of given positive masses deciding if the number of classes of central configurations is finite(see, for instance: \cite{AlbouyCabralSantos,SmaleProblems}). More generally, \textit{the direct problem} consists of fixing the masses and finding some positions, giving a central configuration.

The {\it inverse problem} consists of fixing the positions and finding the masses, if any, making the configuration central. A  classic example is
 Lagrange's  result(\cite{Lagrange}), which states that any masses can form a central
configuration for an equilateral triangle. At any dimension, any masses at the regular simplex give rise to a central configuration(see \cite{SaariOntheRole}).

The direct problem for collinear central configurations was studied in the classical works by Euler (\cite{Euler}) and Moulton
(\cite{Moulton}). The inverse problem in the collinear case was studied by Moeckel and Albouy(\cite{MoeckelAbouy})  and by Candice {\it et al} (\cite{Candice}).

In studying central configurations, the symmetry on the positions is a widely used hypothesis, but there is no uniform approach to exploit symmetry in decomposing the equations. The paper \cite{SANTOS_22} addresses the problem using the representation theory of finite groups. The present paper continues presenting this approach by improving results and presenting applications to new cases. Representation theory was also recently used to study stability problems(see \cite{XIA_ZHOU,LEANDRO,LEANDRO_2}).

In this paper, as an application of the developed theory, we prove the existence of central configurations of two nested regular tetrahedrons, two nested regular octahedrons, and two nested regular cubes. 
In each case, we prove the uniqueness of masses(i.e., equal masses in each polyhedron), give explicit relations for the ratio of masses and the edge of the polyhedra, and prove there is an interval where the size of edges can be chosen. As far as we know, the results for the cube case are completely new.
We invite the reader to compare our results with the results cited below.

In \cite{Corbera_Llibre_2n},
Corbera and Llibre study the existence of central configurations of two nested regular platonic polyhedra. They assume that the masses in each polyhedron are equal. They prove that for each ratio of the masses, there exists a unique ratio between the edges of the polyhedra such that the configuration is central. Later, Corbera and Llibre extended their results to include more polyhedra(see  \cite{Corbera_Llibre_3n} and \cite{Corbera_Llibre_Pn}). 

In \cite{ZHU}, Zhu studies necessary and sufficient conditions for nested regular tetrahedrons to form a central configuration. He proves that the masses in each tetrahedron have to be equal. Additionally, he provides a sufficient condition involving a relationship between the masses and the size of the edges of the tetrahedrons. In \cite{Liu_Tao}, Liu and Tao prove similar results as \cite{ZHU}, but for the case of 
 nested regular octahedrons.

Two other important articles related to these results (which assume equal masses on the orbits) are the articles by Ced\'o and Llibre \cite{Cedo_Llibre}, and Montaldi \cite{Montaldi}. In the first, they study symmetric central configurations with equal masses where the positions form an orbit of a finite group of isometries acting on $\mathbb{R}^3$. In the latter, he shows that there exists at least one central configuration given any symmetric configuration and symmetric distribution of masses.

Montaldi used a variational approach and obtained a general result. Whether there are more central configurations with the same symmetry properties remains an open question. Hence, it is desirable to study the specificity of some particular cases. Therefore, it is helpful to have a tool to simplify the equations for these symmetric cases.
It is also worth noting that the technique presented here allows computer algebra systems to make symbolic calculations related to problems that may otherwise not be attainable at our current hardware.

This paper is organized as follows:
\begin{itemize}
    \item[$\bullet$] In Section \ref{The Equations For Central Configurations}, we write the equations used for central configurations.
    \item[$\bullet$] In Section \ref{The Symmetry on the Central Configuration Equations}, we show how the symmetry of positions affects the equations.
    \item[$\bullet$] In Section \ref{Finding A Symmetry Adapted Basis}, we introduce a prelude to representation theory to show how to compute the change of basis that reduces the equations.
    \item[$\bullet$] In Section \ref{Applications: The inverse and direct problem for central configurations of two nested tetrahedron, two nested octahedron}, we discuss some methods of rational parameterization to analyze equations after the reduction given by the change of basis and analyze the cases of two nested tetrahedra and two nested octahedra.
    \item[$\bullet$] In Section \ref{Applications: The inverse and direct problem for central configurations of two nested cubes}, we present some algebra tools to analyze equations after the reduction given by the change of basis and analyze the case of two nested cubes.
     \item[$\bullet$] In Section \ref{Sec: Conclusion}, we summarize the results in a conclusion.
    \item[$\bullet$] In Section \ref{Data availability}, we provide a link to the calculations in the software Sagemath.
\end{itemize}

\section{The Equations For Central Configurations}
\label{The Equations For Central Configurations}
For our purposes, it is convenient to rewrite the equations of central configurations. And we do it employing the following proposition. Since we consider only positive masses, the total mass $ M=\sum_{j=1}^{N} m_j$ is nonzero.

\begin{proposition}
 Suppose the total mass $ M=\sum_{j=1}^{N} m_j$ is nonzero. Then the following conditions are equivalent
\begin{itemize}
    \item[i)] There exists a constant $\lambda \in \mathbb{R}$ such that
\begin{equation}
-\lambda(q_i- q_0)=\sum_{j \neq i} m_j\frac{q_j-q_i}{|q_j-q_i|^3} \quad ,i=1,\ldots, N.\label{EquacaoDeConfiguracaoCentralProp}
\end{equation} 
\item[ii)] There exists a constant $c \in \mathbb{R}$ 
{\small
\begin{equation}
\sum_{j \neq i} m_j\left(\frac{1}{|q_j-q_i|^3}-c\right)(q_j-q_i)={\bf 0} \quad ,i=1,\ldots, N.\label{EquacaoDeConfiguracaoCentralNewDef}
\end{equation} }
\end{itemize}
\end{proposition}
\begin{proof} $i) \Rightarrow ii)$ Using the expression \eqref{CentroDeMassa} for the center of mass, we conclude that the equation \eqref{EquacaoDeConfiguracaoCentralProp} is equivalent to $\sum_{j \neq i} m_j\left(\frac{1}{|q_j-q_i|^3}-\frac{\lambda}{M}\right)(q_j- q_i)={\bf 0}. $
It suffices to choose $c=\frac{\lambda}{M}$.\\
$ii) \Rightarrow i)$ On the other hand, equation \eqref{EquacaoDeConfiguracaoCentralNewDef}
is equivalent to the equation
{\small
\begin{align*}
 \sum_{j \neq i} m_j\frac{q_j-q_i}{|q_j-q_i|^3}=-c\sum_{j =1}^{N}m_j (q_i-q_j)=-(cM)(q_i-q_0)
.\end{align*}}
So if we choose $\lambda=(cM),$ equation \eqref{EquacaoDeConfiguracaoCentralProp} is satisfied.
\end{proof}

The configuration $q$ is central  if the equation \eqref{EquacaoDeConfiguracaoCentralProp}(respectively \eqref{EquacaoDeConfiguracaoCentralNewDef}) holds for some value of $\lambda$ (respectively $c$). It is possible to show that the values $\lambda$ and $c$ are determined uniquely by the respective equation, the constant $\lambda$ corresponding to the quotient between the potential and the inertia moment, and $c$ corresponding to this quotient divided by the total mass. Thus, we are interested only in solutions of $\eqref{EquacaoDeConfiguracaoCentralNewDef}$ 
such that $c$ is positive.

Given a configuration $q=(q_1,\ldots,q_n)\in\mathbb{R}^{dN}$, define 
its associated matrix, of order $Nd\times N$,   by 
\begin{equation}
S(c)=[S_{ij}(c)], \label{matrixS}
\text{ where }  S_{ij}(c)=\begin{cases}
\left(\frac{1}{|q_j-q_i|^3}-c\right)(q_j-q_i)&\text{ for }  i\neq j,\\
{\bf 0}&  \text{ for } i=j.
\end{cases}
\end{equation}
 Consider
the mass vector $\overrightarrow{m} =(m_1,m_2,\ldots,m_N)$.
By equation \eqref{EquacaoDeConfiguracaoCentralNewDef}, the linear system
\begin{align}
    S(c)\overrightarrow{m}={\bf 0},
\end{align}
depending on the parameter $c$ is equivalent to equations for central configurations.
Therefore, we will study which values of $c$ in the matrix $S(c)$ have a kernel with a vector mass $\overrightarrow{m}\in (\mathbb{R}^{+})^N$.
 In the next section, we omit the parameter $c$.

\section{The Symmetry on the Central Configuration Equations}
\label{The Symmetry on the Central Configuration Equations}

Let $q=(q_1,q_2,\ldots,q_N)\in \mathbb{R}^{Nd}$ be a configuration. And $O(d)$
the group of orthogonal linear transformations on $\mathbb{R}^ d$. Suppose a homomorphism $\rho: G \rightarrow O(d)$
permutes the positions. 

Let $S_N$ be the group of permutations of $N$ elements 
and $\zeta:G\rightarrow S_{N}$ 
the homomorphism induced on the set of indices by the action of $\rho$,
i.e., $\rho(g)(q_i)=q_{\zeta(g)(i)}.$ 
Recalling the expression for $S_{ij}$ given at $\eqref{matrixS}$ and using the fact that $\rho(g)$ is orthogonal, 
it follows that 
\begin{equation}
\rho(g)(S_{ij})=S_{\zeta(g)(i)\zeta(g)(j)}. 
\label{action_rho}
\end{equation}
If we consider the left block-diagonal action of $\rho(g)$ on $S$, we get the matrix $[\widehat{S}_{ij}]$,  where $\widehat{S}_{ij}=S_{\zeta(g)(i)\zeta(g)(j)}.$
More precisely, let $\mathbb{I}_N:G \to GL(\mathbb{R}^N)$ denote the trivial homomorphism and consider the tensor product  $(\mathbb{I}_N\otimes\rho):G\to GL(\mathbb{R}^{Nd})$, so
\begin{equation}
\label{ItensorRho} (\mathbb{I}_N\otimes\rho)(g)S=[\widehat{S}_{ij}] \text{ where } \widehat{S}_{ij}=S_{\zeta(g)(i)\zeta(g)(j)}.
\end{equation}
Let $\theta:G \rightarrow GL(\mathbb{R}^{N})$ be the a homomorphism  given by
\begin{equation}
\theta(g)((m_1,m_2,\ldots,m_N)^T)= (m_{\zeta(g^{-1})(1)},m_{\zeta(g^{-1})(2)},m_{\zeta(g^{-1})(N)})^T.\label{RightActionForTheta}
\end{equation}
The matrix presentation of $\theta$ is given by the permutation matrices $\theta_{ij}(g)=\delta_{i,\zeta(g)(j)}$(where $\delta$ stands for the Kronecker delta). The right multiplication by the matrix $\theta(g)$ results in permuting the columns, while the left multiplication permutes the rows. As $S$ is a matrix of order $Nd\times N$, we act on the left multiplication through the tensor product $\theta\otimes \mathbb{I}_d$, where $\mathbb{I}_d$ represents the identity on $\mathbb{R}^d$.
It is easy to verify that
\begin{equation}
(\theta\otimes \mathbb{I}_d)(g^{-1})S\theta(g)
=\left[\begin{array}{ccc}
S_{\zeta(g)(1)\zeta(g)(1)}&\ldots&S_{\zeta(g)(1)\zeta(g)(N)} \\ 
\vdots&\vdots&\vdots\\
S_{\zeta(g)(N)\zeta(g)(1)}&\ldots&S_{\zeta(g)(N)\zeta(g)(N)} \\ 
\end{array}\right]
\label{TGInvST}.
\end{equation}
Using equations \eqref{ItensorRho} and \eqref{TGInvST}, we may conclude that
\begin{equation*}
\left(\theta\otimes\mathbb{I}_d\right)(g^{-1})S\theta(g)=\left(\mathbb{I}_N\otimes\rho \right)(g)S.
\end{equation*}
Hence, using the properties of the tensor product we conclude that
\begin{equation*}
S\theta(g)=\left(\theta\otimes\mathbb{I}_d\right)(g)\left(\mathbb{I}_N\otimes\rho \right)(g)S=\left(\theta\otimes\rho\right)(g)S.
\end{equation*}
 We summarize the discussion in the next theorem.
\begin{theorem}[Symmetry Theorem]
\label{SymmetryTheorem}
 Let $q=(q_1,q_2,\ldots,q_N) \in \mathbb{R}^{Nd}$ be a configuration and
 suppose that a finite group $G<O(d)$ acts on the configuration $q$, through the homomorphism $\rho$, permuting the positions.
 Consider the permutation action on indices 1,2, ..., N induced by the action of $\rho$. Let $\theta: G \to O(N)$ be the matrix presentation of this action.
 If 
 $S$ is the matrix associated with $q$ defining the central configuration equations \eqref{matrixS}, then the following equation holds
\begin{equation}
S\theta(g)=\left(\theta\otimes\rho\right)(g)S,\quad \forall g \in G. \label{Symmetry_Equation}
\end{equation}
 \end{theorem}
\section{Finding A Symmetry Adapted Basis}
\label{Finding A Symmetry Adapted Basis}

This section gives a concise outline of the symmetry-adapted basis method of representation theory of finite groups. 
For brevity, we will omit demonstrations. The interested reader can find the demonstrations and detailed discussions in the good references \cite{Serre}(Chapters 1 to 5) and \cite{Stiefel_Fassler}.

Let $V$ be a finite-dimensional vector space over $\mathbb{C}$. Let $G$ be a finite group. A homomorphism $\phi:G\to GL(V)$ is a {\it linear representation} of $G$. 
When $\phi$ is given, we say that $V$ is a representation of $G$.
The dimension of $V$ is called the {\it degree} of $\phi$, we shall denote the degree of $\phi$ as $\deg(\phi)$.
A vector subspace $U\subset V$ is said $\phi$-invariant if $\phi(g)(U)\subset U$ for all $g \in G$. A $\phi$-invariant subspace $U\subset V$ is a {\it subrepresentation} of $\phi$. So, a subrepresentation is by itself also a representation.
A representation $\phi$ is {\it reducible} if there is a $\phi$-invariant proper subspace $U\subset V$. And {\it irreducible} otherwise. 
\begin{theorem}[Maschke's Theorem]
If there is a $\phi$-invariant subspace $V_1\subset V$.  Then there is also a complementary $\phi$-invariant subspace $V_2\subset V$ such that $V=V_1\oplus V_2$.
\end{theorem}
\begin{proof}
See \cite[p.6]{Serre}. 
\end{proof}

As a corollary of Maschke's theorem, we conclude that a representation $\phi$
is completely reducible, which means that there is a decomposition $V=\bigoplus_{i=1}^{r} U_i$ such that the subrepresentation $\phi_i(g)=\phi(g)|_{U_i},\forall g \in G$ is irreducible. Accordingly, we can express 
$ \phi=\phi_{1}\oplus\cdots \oplus\phi_{r}.$
Since $U_i$ is invariant, choosing a basis for each $U_i$ to form a basis for $V,$ the matrix presentation of $\phi$ is given by block diagonal matrices, where each block corresponds to the restriction to $U_i$.
This decomposition implies the existence of a change of basis $P$ such that the representation is given by
\begin{align}
\label{Decomposing_a_representation_into_blocks}
    P^{-1}\phi(g)P= \left (
\begin{array}{r|r|r}
\phi_1(g) &  & \\\hline
    & \ddots & \\\hline
    &  & \phi_r(g) 
\end{array}\right).
\end{align}
Given two representations $\phi:G\to GL(V)$ and $\psi:G \to GL(W)$ a linear transformation $\tau:V\to W$ such that 
\begin{align}
\psi(g)\circ \tau= \tau\circ \phi(g), \quad\forall g \in G,     
\end{align}
is called an \textit{equivariant map}. If $\tau$ is  an isomorphism, we say $\psi$ and $\phi$ are \textit{equivalent}.

Consider the decomposition of a representation $\phi: G\to GL(V)$ into its irreducible subrepresentations. If we collect all equivalent subrepresentations, then we have
\begin{align}
V_j=\bigoplus_{l=1}^{c_j} V_{j}^{l},\label{isotypic_component}
\end{align}
where $V_{j}^{l}$ are equivalent subrepresentations for $l=1,\ldots, c_j.$ 
The number $c_j$ is called {\it multiplicity}. The degree of $V_j^l, l=1,\ldots,c_j$ is the same since they are equivalent, and it is denoted by $n_j$.
The subspaces $V_j$ are called {\it isotypic components}.
So, the isotypic component $V_j$ is a $c_jn_j$-dimensional subspace.
 The decomposition of 
 $V_j$ pointed out in \eqref{isotypic_component} is not unique.
 On the other hand, the decomposition of the whole space \begin{align}
\label{isotypic_decomposition}
  V=V_{1}\oplus\cdots \oplus V_{n},  
\end{align}
induced by the isotypic components is called {\it isotypic decomposition}, and it is unique.
The decomposition of an equivariant map based on the isotypic decomposition was used in \cite{SANTOS_22}. In what follows, we explore a more refined decomposition.

The decomposition into equivalent irreducible  subrepresentations,
$
V_j=\bigoplus_{l=1}^{c_j} V_{j}^{l},
$ induces a
diagonal block composed of equivalent subrepresentations. We will represent this block as
{\small
\begin{align}
\label{block_jl_in_the_matricial_presentation}
\def\arraystretch{1.4}
    c_j\phi_j=\left(\begin{array}{r|r|r}
\phi_j^1 &        &  \\\hline
       &\ddots  &  \\\hline
       &        &\phi_j^{c_j} \end{array}\right).
\end{align}
}
If necessary, after reordering, we can assume that in equation \eqref{Decomposing_a_representation_into_blocks}, the equivalent representations are grouped, with blocks given by \eqref{block_jl_in_the_matricial_presentation} and its matricial presentation is given by 
{\small
\begin{align}
\label{Decomposing_cj_phij}
\def\arraystretch{1.4}
    P^{-1}\phi P= \left (
\begin{array}{r|r|r|r}
c_1\phi_1 & & & \\\hline
          &c_2\phi_2 & & \\\hline
          & & \ddots & \\\hline
          & &   & c_n\phi_n
\end{array}\right).
\end{align}
}
Correspondingly to this decomposition \eqref{Decomposing_cj_phij}, we can write 
\begin{align}
\label{Decomposition_multiplicity_times_irreducible}    \phi=c_1\phi_{1}\oplus\cdots \oplus c_n\phi_{n}.
\end{align}
\subsection{The existence of a block matrix structure for an equivariant map}
\label{block_structure_equivariant_operator}
It is well-known in the literature that
we can decompose an equivariant map  into a block matrix structure with many null blocks. We present some theorems that imply the existence and a way to calculate the block structure.
The main result is the following lemma. 
\begin{theorem}[Schur's Lemma]\label{lema_de_schur}
Let $\psi: G \longrightarrow GL(V)$ and $\phi: G \longrightarrow GL(W)$ be two irreducible representations of the group $G$. And $\tau$ an equivariant map.
Exactly one of the following affirmatives is true
\begin{itemize}
    \item[a)] $\tau$ is the null linear transformation. 
    \item[b)] $\tau$ is a linear isomorphism, $\psi$ and $\phi$ are equivalent and $\psi(g)=\tau\circ\phi(g)\circ\tau^{-1}$.
\end{itemize}
\end{theorem}
\begin{proof} See \cite[p.~13]{Serre} or \cite[p.~33]{Stiefel_Fassler}.
\end{proof}
In order to use Schur's Lemma, we may decompose the representations in their irreducible constituents, as explained in the following lemma.
\begin{lemma}
\label{lemma_block_zero}
Let $\psi:G\rightarrow GL(V)$ and $\phi:G\rightarrow GL(W)$ be two representations that decompose into irreducible as
$\psi=b_1\psi_{1}\oplus\cdots \oplus b_k\psi_{k}$ and $\phi=c_1\phi_{1}\oplus\cdots \oplus c_n\phi_{n}$. Suppose $\tau: V \rightarrow W$ is an equivariant map. There exists a decomposition of $\tau$ into blocks with the property: 
If $\psi_i$  is not equivalent to $\phi_j$,
there is a corresponding null block $\mathcal{T}_{ij}$.
This block is located
between rows $1+\sum_{p<i}b_p\cdot\deg(\psi_p)$ and $\sum_{p<i+1}b_p\cdot\deg(\psi_p)$; and columns $1+\sum_{p<j}c_p\cdot\deg(\phi_p)$ and $\sum_{p<j+1}c_p\cdot\deg(\phi_p)$.
 \end{lemma}

\begin{proof}

There exists a change of basis matrices $P$ and $Q$ such that $\overline{\phi}=P^{-1}\phi P$ and $\overline{\psi}=Q^{-1}\psi Q$ have a form alluded at the equation \eqref{Decomposing_cj_phij}.
Since
$\psi\circ \tau= \tau\circ \phi$, then $ Q\overline{\psi}Q^{-1} \tau= \tau P\overline{\phi}P^{-1}\Rightarrow    \overline{\psi}Q^{-1} \tau P= Q^{-1}\tau P\overline{\phi}$, set
$\mathcal{T}=Q^{-1}\tau P$ the expression of $\tau$ relative to the new basis, then the following equation holds
\begin{align}
 \overline{\psi} \mathcal{T}= \mathcal{T} \overline{\phi},\label{equivariance_of_tau}
\end{align}
Subdivide $\mathcal{T}$ into blocks in the following form:
For each $i=1, \ldots, k$ group a number of rows equal to $b_i$ times the degree of $\psi_i$;  and for each $j=1,\ldots, n$
    group a number of columns equal to $c_j$ times the degree of $\phi_j$. Then we get
\begin{align}
\mathcal{T} =\left(\begin{array}{c|c|c}
    \mathcal{T}_{11} & \ldots & \mathcal{T}_{1n}\\\hline
    \mathcal{T}_{21} & \ldots & \mathcal{T}_{2n}\\\hline
    \vdots & \ldots & \vdots\\\hline
    \mathcal{T}_{k1} & \ldots &  \mathcal{T}_{kn}
    \end{array}\right).
    \label{first_subdivision}
\end{align}
By equation \eqref{equivariance_of_tau}, the block $\mathcal{T}_{ij}$ satisfies the equation 
\begin{align}
\label{equivariance_on_blocks}
b_i\psi_i\mathcal{T}_{ij}=\mathcal{T}_{ij}c_j\phi_j.
\end{align}
We shall note that the block  $\mathcal{T}_{ij}$ is obtained from $\mathcal{T}$ by selecting rows between $1+\sum_{p<i}b_p\deg(\psi_p)$ and $\sum_{p<i+1}b_p\deg(\psi_p)$; columns between $1+\sum_{p<j}c_p\deg(\phi_p)$ and $\sum_{p<j+1}c_p\deg(\phi_p)$.

Subdividing $\mathcal{T}_{ij}$ into blocks with a number of rows equal to the degree of $\psi_i$ and the number of columns equal to the degree of $\phi_j$, the equation \eqref{equivariance_on_blocks} become
{\small
\begin{align}
\arraycolsep=0.1pt\def\arraystretch{1.3}
\left (
\begin{array}{r|r|r}
\psi_i^1 &        &  \\\hline
       &\ddots  &  \\\hline
       &        &\psi_i^{b_i} 
\end{array}
 \right)
\left(\begin{array}{r|r|r}
    T^{ij}_{11} & \ldots & T^{ij}_{1c_j}\\\hline
    T^{ij}_{21} & \ldots & T^{ij}_{2c_j}\\\hline
    \vdots & \ldots & \vdots\\\hline
   T^{ij}_{b_i1} & \ldots &  T^{ij}_{b_ic_j}\\
     \end{array}\right)
= 
\left(\begin{array}{r|r|r}
    T^{ij}_{11} & \ldots & T^{ij}_{1c_j}\\\hline
    T^{ij}_{21} & \ldots & T^{ij}_{2c_j}\\\hline
    \vdots & \ldots & \vdots\\\hline
   T^{ij}_{b_i1} & \ldots &  T^{ij}_{b_ic_j}\\
     \end{array}\right)
\left (
\begin{array}{r|r|r}
\phi_j^1 &        &  \\\hline
       &\ddots  &  \\\hline
       &        &\phi_j^{c_1} 
\end{array}
\right).
\label{matricial_equivariance}
\end{align}
}
which implies
\begin{align}
    \psi_i^{l}{T}^{ij}_{lm}={T}^{ij}_{lm}\phi_j^{m}.
\end{align}
By Schur's Lemma, if $\psi_i$ and $\phi_j$ are not equivalent, then each sub-block ${T}^{ij}_{lm}$ is zero. Hence $\mathcal{T}_{ij}$ is zero.$\blacksquare$ 
\end{proof}

In the notation of Lemma \ref{lemma_block_zero},
if $\psi_i$ and $\phi_j$ are equivalent, we can have a nonnull block $\mathcal{T}_{ij}$. This block can be further decomposed into equal sub-blocks as described below. See the following lemma.

\begin{lemma}
\label{lema_decompoe_mais}
In the notation of Lemma  \ref{lemma_block_zero}, if $\psi_i$ is equivalent to $\phi_j$, the block
$\mathcal{T}_{ij}$, can be decomposed into the form
{\small
\begin{align}
\label{form_repeting_blocks}
\mathcal{T}_{ij}=\left (
\begin{array}{r|r|r}
\mathfrak{t}_{ij} &        &  \\\hline
       &\ddots  &  \\\hline
       &        &\mathfrak{t}_{ij} 
\end{array}
\right),
\end{align}
}
with $\deg(\phi_j)$ equal blocks $\mathfrak{t}_{ij}$ of $b_i$ rows and $c_j$ columns. 
\end{lemma}

We will omit the demonstration of Lemma \ref{lema_decompoe_mais}.
The details about the existence of the base giving the form at Lemmas \ref{lemma_block_zero} and \ref{lema_decompoe_mais}  are described in the reference \cite{Stiefel_Fassler}. 
In the next section, we will give a quick guide on how to build the change of variables.

\subsection{Calculating the change of basis}
Let $\chi_{\phi}(g)=trace({\phi}(g))$. The function $\chi:G\to \mathbb{C}$ is called the {\it character } of $\phi$. The inner product between characters is given by 
\begin{align}
\label{Produto_interno_dos_Carateres}
    ( \chi_1, \chi_2 )=\frac{1}{|G|}\sum_{g\in G} \chi_1(g^{-1})\cdot \chi_2(g).
\end{align}    
    The inner product $( \chi_{\phi}, \chi_i )$ is equal to the multiplicity $c_i$, such that one irreducible representation of character $\chi_i$ occurs as a subrepresentation of $\phi$. In particular, given two irreducible representations of a group $G$, they are equivalent if, and only if, their inner product is 1. Therefore, the character provides a way to identify irreducible representations. The character of an irreducible representation is called an {\it irreducible character}.
    
    Consider a finite group $G$; it is possible to construct the complete set $\{\chi_1,\ldots,\chi_k\}$
    of irreducible characters.
  The projection of the representation $\phi$ over the isotopic component $V_j$ given in equation \eqref{isotypic_component} is 
\begin{align}
P_j=\sum_{g \in G} \chi_j(g^{-1})\phi(g).
\label{isotypicDecomposition}
\end{align}

Suppose a complete set of nonequivalent irreducible representations of a group  it is given in a matrix form. We can use these matrices to further decompose $V_j$ into its irreducible components $V_j^{l}, l=1,\ldots,c_j$ like in \eqref{isotypic_component}. If $\delta_j$ is a representation equivalent to $\phi_j,$ and  it is given in matrix form by $\delta_j(g)=(d^j_{ki}(g))$, and let $n_j$ be the degree of  $\phi_j$, define the operators
\begin{align}
P^{j}_{ki}=\frac{n_j}{|G|}\sum_{g \in G} d^{j}_{ki}(g^{-1})\phi(g). \label{transferences}
\end{align}

The linear operator $P^j_{ki}$ is zero on $V_l$ for $l \neq j$ and defines an isomorphism between 
$V_j^k$ and $V_j^i$. The operators $P^j_{ki}$ are sometimes called {\it transferences}.
 Using the transferences, we will construct the so-called {\it symmetry-adapted basis}.
 One can fully decompose the representation in its irreducible components with the symmetry-adapted basis.
Also, if an equivariant operator exists, one can decompose into blocks by reordering this basis.

 To find a symmetry-adapted basis, follow the steps:
 \begin{itemize}
    \item[i)]     The operator $P^j_{11}$ has  a rank equal to
    the multiplicity $c_j$. Choose a basis $\beta_j^1=\{v_1^1,v_1^2,\ldots,v_1^{c_j}\}$ of its image.
    \item[ii)] For $k=2,\ldots,n_j$, and $i=1,2,\ldots,c_j$ calculate the vectors
    \begin{align}
      \label{Equation_for_transferences}  v_k^i=P^j_{1k}v_1^i.
    \end{align}
     The set
    $\beta_j^i=\{v_1^i,v_2^i,\ldots,v_{n_j}^{i}\}$ is a ordered basis of $V_j^i$, for $i=1,2,\ldots,c_j$.
    \item[iii)] By juxtaposing the basis of the previous step, we obtain $\beta_j=\beta_j^1\cup \beta_j^2\cup\ldots\cup \beta_j^{c_j}$, a basis for the isotypic component $V_j$.  Let $\beta=\beta_1\cup\beta_2\cup\ldots\cup\beta_k$. The matrix presentations of $\phi$  in the basis $\beta$ have the block diagonal form 
        
    \begin{align}
\phi^{\beta}=\begin{pmatrix}
   c_1\phi_1 & & \\
    & \ddots & \\
     & & c_n\phi_n 
        \end{pmatrix}.
    \end{align}
       \end{itemize}
The basis that decomposes an equivariant operator in the form alluded to in Lemmas \ref{lemma_block_zero} and \ref{lema_decompoe_mais} is the same basis presented here but in a different order. The order basis for the equivariant operator is 
\begin{align}
\tilde{\beta}_{j}^{k}=\{v_k^1, v_k^2, \ldots, v_k^{c_j}\}, \quad k=1,\ldots,n_j. 
\end{align}
Therefore $\beta^{i}_j$ is a basis for the $n_j$-dimensional subspace $V^i_j$, for $i=1,\ldots,c_j$. On the other hand, $\tilde{\beta}^{k}_j$ is a basis for the $c_j$-dimensional subspace $U^k_j$ for $k=1,\ldots,n_j$. In any case, $V_j=\bigoplus_{l=1}^{c_j} V_{j}^{l}=\bigoplus_{k=1}^{n_j} U_{j}^{k}$. When choosing the subspaces $U_{j}^{k}$ the restrictions of an equivariant operator is zero when the subrepresentations are not equivalent or have the form of Lemma \ref{lema_decompoe_mais} when they are equivalent.
In Figure \ref{order_basis_diagram} we give a schematic representation of the order of the bases(inspired by the reference \cite{Stiefel_Fassler}).
\begin{figure}[H]
\begin{tikzpicture}[descr/.style={fill=white,inner sep=2pt}]
        \matrix (m) [
            matrix of math nodes,
            row sep=1em,
            column sep=1.5em,
            text height=1.5ex, text depth=0.25ex
        ]
        { V_j^{1} & v_1^1 & v_2^1  & \ldots & v_{n_j}^1 &  & v_1^1 & v_2^1  & \ldots & v_{n_j}^1 \\
           V_j^{2} & v_1^2 & v_2^2  & \ldots & v_{n_j}^2 &  & v_1^2 & v_2^2  & \ldots & v_{n_j}^2 \\
                   &\vdots&        &   & \vdots & &   &        &  & \mbox{}      \\
           V_j^{c_j} & v_1^{c_j} & v_2^{c_j}  & \ldots & v_{n_j}^{c_j} & &  v_1^{c_j} & v_2^{c_j}  & \ldots & v_{n_j}^{c_j} \\
           \mbox{} & \mbox{} & \mbox{}  & \mbox{} & \mbox{} & &  U_j^{1} & U_j^{2}  & \ldots & U_j^{n_j} \\
        };
        \path[overlay,->, font=\scriptsize,>=latex]
        (m-1-2) edge (m-1-3)
        (m-1-3) edge (m-1-4)
        (m-1-4) edge  (m-1-5)
        (m-1-5) edge[out=255,in=055]  (m-2-2)
        (m-2-2) edge (m-2-3)
        (m-2-3) edge (m-2-4)
        (m-2-4) edge (m-2-5)
        (m-2-5) edge[out=255,in=055]  (m-3-2)
        (m-4-2) edge (m-4-3)
        (m-4-3) edge (m-4-4)
        (m-4-4) edge (m-4-5);
              \path[->, font=\small,>=latex]
        (m-1-7) edge[dashed] (m-2-7); 

         \path[->, font=\small,>=latex]
        (m-2-7) edge[dashed] (m-4-7); 

         \path[->, font=\small,>=latex]
        (m-4-7) edge[dashed][out=355,in=175] (m-1-8); 
         \path[->, font=\small,>=latex]
        (m-1-8) edge[dashed] (m-2-8); 

         \path[->, font=\small,>=latex]
        (m-2-8) edge[dashed] (m-4-8); 

         \path[->, font=\small,>=latex]
        (m-4-8) edge[dashed][out=355,in=175] (m-1-9);
        \path[->, font=\small,>=latex]
        (m-4-9) edge[dashed][out=355,in=175] (m-1-10); 
         \path[->, font=\small,>=latex]
        (m-1-10) edge[dashed] (m-2-10); 

        \path[->, font=\small,>=latex]
        (m-2-10) edge[dashed] (m-4-10);

\end{tikzpicture}
     \caption{The left diagram shows the order to obtain the decomposition of the representation by the subspaces $V_j^i$. The right diagram shows the order to obtain the decomposition of the equivariant operator by the subspaces $U_j^k$.}
    \label{order_basis_diagram}
\end{figure}
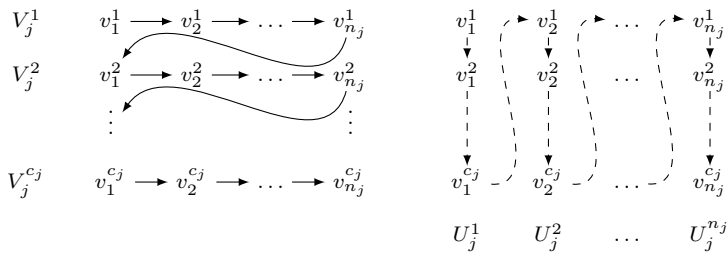

\section{Applications: The inverse and direct problem for central configurations of two nested tetrahedron, two nested octahedron}
\label{Applications: The inverse and direct problem for central configurations of two nested tetrahedron, two nested octahedron}
  
To find the kernel of $\tau$ with $\mathcal{T}=Q^{-1}\tau P$, find
 the vector solution for the kernel of $\mathcal{T}$ and use the entries as  the coordinates for the ordered basis whose vectors are the columns of $P$. If $\mathcal{T}$ has a blockwise decomposition, the systems to solve are smaller and can be solved for each block. Using these tools, we will find the kernel of the matrix $S$ in \eqref{matrixS}, which defines the central configurations equations.

For simplicity of notation, we will denote equivalent subrepresentations with the same index. So $\theta_i$ will be equivalent to $(\theta\otimes\rho)_i$, and as a consequence, a nonnull block $\mathcal{T}_{ij}$ of the equivariant operator will have $i=j$, we will index these blocks with just one index. Furthermore, this index will be chosen according to the description of the set of irreducible sub-representations associated with the group.

After the change of variables given by the symmetry-adapted basis, we analyze the blocks transforming the expressions in rational functions. This analysis allows us to know the coordinates of the mass vector relative to the new basis. Hence, we determine the possible masses. Also, it is possible to determine relations between the masses and the size of the edges of the polyhedra.

We present the case of nested tetrahedron  in detail, with the calculations and theorems. 
The proofs for the other cases are very similar. Hence we will omit the repetition of some arguments when there is no loss of clarity. Also, as the results are very similar in all cases, we will summarize the discussion of this section in some general theorems at the final of this section.

In general, we can perform numerical calculations to speculate about the structure of solutions. However, doing exact calculations through symbolic computations has a much higher computational cost, which can become an impediment in many cases. In Table \ref{Order_of_groups_and_matrices}, we present the orders of matrices involved in calculations.
\begin{table}

\begin{center}

\begin{tabular}{l|c|c|c|c|c}
& \text{Group} & \text{Group order} & \text{Order of $S$} & \text{Degree of $\theta$} & \text{Degree of $\theta\otimes\rho$} \\[2pt]\hline
\text{Tetrahedron} & $S_4$ & $24$ &  $24\times 8$ & $8$ & $24$ \\\hline
 \text{Octahedron} & $O_h$ & $48$ &   $36\times 12$ & $12$ & $36$    \\\hline
\text{Cube} & $O_h$ & $48$ & $48\times 16$ & $16$  & $48$ \\ \hline
\end{tabular}
\end{center}
\caption{$S$ represents the matrix defining the central configuration equations. The order of the group is the number of matrices we must calculate to do the change of basis. The degree is the order of these matrices.}
\label{Order_of_groups_and_matrices}
\end{table}

We consider central configurations of two nested platonic polyhedra.
Let us assume that one polyhedron has vertices $q_j$, for $j=1,\ldots,n$, and the other has vertices $q_{j+n}=tq_j$ where $t \in \mathbb{R}$ and $j=1,\ldots,n$. So, $t$ is the ratio between the edges of the two polyhedra. For simplicity, we only give the positions $q_1,\ldots, q_n$ in what follows.
The equations  \eqref{EquacaoDeConfiguracaoCentralNewDef} are invariant by dilatation, so we can assume without loss of generality that  $t\in (0,1)$. So the positions $q_1,\ldots,q_n$ are vertices for the outer polyhedron.

\subsection{Nested Tetrahedra}
In this subsection, we reobtain the results of \cite{ZHU}, concerning homothetic tetrahedra(see Theorem 3 in that reference). Also, we use this case to explain the method, giving many details in the calculations.

Let the positions of the outer regular tetrahedron be $q_1=(1,1,1)^T,$ $q_2=(-1,-1,1)^T,$ $q_3=(-1,1,-1)^T,$ $q_4=(1,-1,-1)^T$. The group of symmetries of the tetrahedron is $S_4$(the symmetric group of four elements).
In Table \ref{Representacoes_irredutiveis_S4}, the generators of this group appear in the first row. The irreducible nonequivalent representations appear in the first column. These irreducible nonequivalent subrepresentations are available in the literature or can be accessed via software like GAP(see \cite{GAP4}).
\begin{table}
\begin{center}
\caption{Irreducible subrepresentations of $S_4$}
\label{Representacoes_irredutiveis_S4}
\begin{tabular}{l|c|c|c}
& $(3,4)(7,8)$ & $(2,3)(6,7)$ & $(1,2)(3,4)(5,6)(7,8)$\rule{0pt}{8pt} \\[2pt]\hline
$r_1$ & $\begin{pmatrix}
1 
\end{pmatrix}$ & $\begin{pmatrix}
1 
\end{pmatrix}$ &  $\begin{pmatrix}
1 
\end{pmatrix}$ \rule{0pt}{10pt} \\[2pt]\hline
 $r_2$ & $\begin{pmatrix}
-1 
\end{pmatrix}$ & $\begin{pmatrix}
-1 
\end{pmatrix}$ &  $\begin{pmatrix}
1 
\end{pmatrix}$ \rule{0pt}{10pt}\\[3pt]\hline
$r_3$ & \scriptsize{$\begin{pmatrix}
0 &1 \\
1 & 0
\end{pmatrix}$} & \scriptsize{$\begin{pmatrix}
0 &-\frac{1}{2}+\frac{\sqrt{3}i}{2} \\
-\frac{1}{2}-\frac{\sqrt{3}i}{2} & 0
\end{pmatrix}$}  &  \scriptsize{$\begin{pmatrix}
1 & 0\\
0 & 1
\end{pmatrix}$}\rule{0pt}{16pt} \\[10pt]\hline
$r_4$ & \scriptsize{$\begin{pmatrix}
0 & 0 & 1 \\
0 & 1 & 0 \\
1 & 0 & 0
\end{pmatrix}$} & \scriptsize{$\begin{pmatrix}
1 & 0 & 0 \\
0 & 0 & 1 \\
0 & 1 & 0
\end{pmatrix}$} &  \scriptsize{$\begin{pmatrix}
-1 & 0 & 0 \\
0 & 1 & 0 \\
0 & 0 & -1
\end{pmatrix}$} \rule{0pt}{18pt} \\[12pt]\hline
$r_5$ & \scriptsize{$\begin{pmatrix}
0 & 0 & -1 \\
0 & -1 & 0 \\
-1 & 0 & 0
\end{pmatrix}$} & \scriptsize{$\begin{pmatrix}
-1 & 0 & 0 \\
0 & 0 & -1 \\
0 & -1 & 0
\end{pmatrix}$} &  \scriptsize{$\begin{pmatrix}
-1 & 0 & 0 \\
0 & 1 & 0 \\
0 & 0 & -1
\end{pmatrix}$} \rule{0pt}{19pt}
\end{tabular}
\end{center}
\end{table}

\subsubsection{Obtaining the block decomposition}
The representation $\theta:S_4 \rightarrow GL(\mathbb{C}^8)$
is explicitly calculated at the group generators in Table \ref{Tabela_theta_Tetraedro}.
\begin{table}
\caption{Representation $\theta$ for the tetrahedron case}
\label{Tabela_theta_Tetraedro}
\begin{center}
\begin{tabular}{l|c|c|c}
& $(3,4)(7,8)$ & $(2,3)(6,7)$ & $(1,2)(3,4)(5,6)(7,8)$\rule{0pt}{8pt} \\[2pt]\hline
$\theta$ & \tiny{$\left(\begin{array}{rrrrrrrr}
1 & 0 & 0 & 0 & 0 & 0 & 0 & 0 \\
0 & 1 & 0 & 0 & 0 & 0 & 0 & 0 \\
0 & 0 & 0 & 1 & 0 & 0 & 0 & 0 \\
0 & 0 & 1 & 0 & 0 & 0 & 0 & 0 \\
0 & 0 & 0 & 0 & 1 & 0 & 0 & 0 \\
0 & 0 & 0 & 0 & 0 & 1 & 0 & 0 \\
0 & 0 & 0 & 0 & 0 & 0 & 0 & 1 \\
0 & 0 & 0 & 0 & 0 & 0 & 1 & 0
\end{array}\right)$} & \tiny{$\left(\begin{array}{rrrrrrrr}
1 & 0 & 0 & 0 & 0 & 0 & 0 & 0 \\
0 & 0 & 1 & 0 & 0 & 0 & 0 & 0 \\
0 & 1 & 0 & 0 & 0 & 0 & 0 & 0 \\
0 & 0 & 0 & 1 & 0 & 0 & 0 & 0 \\
0 & 0 & 0 & 0 & 1 & 0 & 0 & 0 \\
0 & 0 & 0 & 0 & 0 & 0 & 1 & 0 \\
0 & 0 & 0 & 0 & 0 & 1 & 0 & 0 \\
0 & 0 & 0 & 0 & 0 & 0 & 0 & 1
\end{array}\right)$} &  \tiny{$\left(\begin{array}{rrrrrrrr}
0 & 1 & 0 & 0 & 0 & 0 & 0 & 0 \\
1 & 0 & 0 & 0 & 0 & 0 & 0 & 0 \\
0 & 0 & 0 & 1 & 0 & 0 & 0 & 0 \\
0 & 0 & 1 & 0 & 0 & 0 & 0 & 0 \\
0 & 0 & 0 & 0 & 0 & 1 & 0 & 0 \\
0 & 0 & 0 & 0 & 1 & 0 & 0 & 0 \\
0 & 0 & 0 & 0 & 0 & 0 & 0 & 1 \\
0 & 0 & 0 & 0 & 0 & 0 & 1 & 0
\end{array}\right)$} \rule{0pt}{33pt} 
\end{tabular}
\end{center}
\end{table}To decompose this representation, we first apply the formula \ref{Produto_interno_dos_Carateres} to obtain the multiplicities. We get
$(\theta,r_i)=2$  for $i=1,4,$ and zero otherwise. The decomposition into irreducible
subrepresentations is  
\begin{align}
\label{Theta_decomposition_Tetrahedron}
 \theta&=2\theta_1\oplus 2\theta_4,
\end{align}
where the index $i$ implies equivalence to $r_i$ in Table \ref{Representacoes_irredutiveis_S4}.

We calculate the symmetry-adapted basis using the \textit{transferences} in formula\eqref{transferences}, with for $\phi=\theta$, and
$d^{j}_{ki}$ is given by Table \ref{Representacoes_irredutiveis_S4}. More details about the calculation are given below.

First, we use in Table \ref{Representacoes_irredutiveis_S4} to obtain $d^{1}_{11}$ the entry of $r_1$, next we calculate $P^1_{11}$ and obtain
\begin{align}
    P^{1}_{11}=\left(\begin{array}{rrrrrrrr}
6 & 6 & 6 & 6 & 0 & 0 & 0 & 0 \\
6 & 6 & 6 & 6 & 0 & 0 & 0 & 0 \\
6 & 6 & 6 & 6 & 0 & 0 & 0 & 0 \\
6 & 6 & 6 & 6 & 0 & 0 & 0 & 0 \\
0 & 0 & 0 & 0 & 6 & 6 & 6 & 6 \\
0 & 0 & 0 & 0 & 6 & 6 & 6 & 6 \\
0 & 0 & 0 & 0 & 6 & 6 & 6 & 6 \\
0 & 0 & 0 & 0 & 6 & 6 & 6 & 6
\end{array}\right).
\end{align}

The subspace $V_{1}$ has dimension $2$, equal to the multiplicity times the degree of $\theta_1$.  We choose the first and fifth columns of  $P^{1}_{11}$to be a base of $U^{1}_{1}$. As the degree is $1$, no other transferences are needed.

To find the decomposition relative to component irreducible $\theta_4$,
we choose $d^{4}_{11}$  the correspondent entry in $r_4$, after calculating $P^{4}_{11},$ we obtain 
\begin{align}
P^{4}_{11}=\left(\begin{array}{rrrrrrrr}
2 & -2 & -2 & 2 & 0 & 0 & 0 & 0 \\
-2 & 2 & 2 & -2 & 0 & 0 & 0 & 0 \\
-2 & 2 & 2 & -2 & 0 & 0 & 0 & 0 \\
2 & -2 & -2 & 2 & 0 & 0 & 0 & 0 \\
0 & 0 & 0 & 0 & 2 & -2 & -2 & 2 \\
0 & 0 & 0 & 0 & -2 & 2 & 2 & -2 \\
0 & 0 & 0 & 0 & -2 & 2 & 2 & -2 \\
0 & 0 & 0 & 0 & 2 & -2 & -2 & 2
\end{array}\right).
\end{align}
Again, by choosing the first and fifth columns, we have a basis for $U_4^{1}$. 
We label these vectors, respectively,
$v_1^1$ and $v_1^2$.

For the transferences $P^4_{12}$ and $P^4_{13},$ we have 
\begin{align}
 P^{4}_{12}=\left(\begin{array}{rrrrrrrr}
2 & -2 & -2 & 2 & 0 & 0 & 0 & 0 \\
2 & -2 & -2 & 2 & 0 & 0 & 0 & 0 \\
-2 & 2 & 2 & -2 & 0 & 0 & 0 & 0 \\
-2 & 2 & 2 & -2 & 0 & 0 & 0 & 0 \\
0 & 0 & 0 & 0 & 2 & -2 & -2 & 2 \\
0 & 0 & 0 & 0 & 2 & -2 & -2 & 2 \\
0 & 0 & 0 & 0 & -2 & 2 & 2 & -2 \\
0 & 0 & 0 & 0 & -2 & 2 & 2 & -2
\end{array}\right)
,\quad    P^{4}_{13}=\left(\begin{array}{rrrrrrrr}
2 & -2 & -2 & 2 & 0 & 0 & 0 & 0 \\
-2 & 2 & 2 & -2 & 0 & 0 & 0 & 0 \\
2 & -2 & -2 & 2 & 0 & 0 & 0 & 0 \\
-2 & 2 & 2 & -2 & 0 & 0 & 0 & 0 \\
0 & 0 & 0 & 0 & 2 & -2 & -2 & 2 \\
0 & 0 & 0 & 0 & -2 & 2 & 2 & -2 \\
0 & 0 & 0 & 0 & 2 & -2 & -2 & 2 \\
0 & 0 & 0 & 0 & -2 & 2 & 2 & -2
\end{array}\right).
\end{align}
Then applying the formula \eqref{Equation_for_transferences} 
we obtain
 $\{v_2^1=P_{12}^4v_1^1, 
v_2^2=P_{12}^4v_1^2\}$ a base for $U_4^{2}$.
Similarly, $\{v_3^1=P_{13}^4v_1^1, v_3^2=P_{13}^4v_1^2\}$ is a base for $U_4^{3}$.

The 
change of the basis matrix is  
\begin{align}
\label{Matrix_p_tetraedro}
P=\left(\begin{array}{rr|rr|rr|rr}
6 & 0 & 2 & 0 & 16 & 0 & 16 & 0 \\
6 & 0 & -2 & 0 & 16 & 0 & -16 & 0 \\
6 & 0 & -2 & 0 & -16 & 0 & 16 & 0 \\
6 & 0 & 2 & 0 & -16 & 0 & -16 & 0 \\
0 & 6 & 0 & 2 & 0 & 16 & 0 & 16 \\
0 & 6 & 0 & -2 & 0 & 16 & 0 & -16 \\
0 & 6 & 0 & -2 & 0 & -16 & 0 & 16 \\
0 & 6 & 0 & 2 & 0 & -16 & 0 & -16
\end{array}\right).
\end{align}
The subdivisions indicate the bases of subspaces $U_1^1, U_4^1,  U_4^2,$ and $
U_4^3$.
The Table \ref{Irreducible subrepresentions of 1 and 4 $S_4$}  shows the matricial expression of $\theta_1$ and $\theta_4$ after the change of basis
given by the correspondents subspaces $V_j^i$, obtained from $U_j^{k}$.
\begin{table}
\begin{center}
\caption{Irreducible subrepresentations of $\theta_1$ and $\theta_4$.}
\label{Irreducible subrepresentions of 1 and 4 $S_4$}
\begin{tabular}{l|c|c|c}
& $(3,4)(7,8)$ & $(2,3)(6,7)$ & $(1,2)(3,4)(5,6)(7,8)$\rule{0pt}{8pt} \\[2pt]\hline
$\theta_1$ & $\begin{pmatrix}
1 
\end{pmatrix}$ & $\begin{pmatrix}
1 
\end{pmatrix}$ &  $\begin{pmatrix}
1 
\end{pmatrix}$ \rule{0pt}{10pt} \\[2pt]\hline
 $\theta_4$ & \scriptsize{$\left(\begin{array}{rrr}
0 & 0 & 8 \\
0 & 1 & 0 \\
\frac{1}{8} & 0 & 0
\end{array}\right)$} & \scriptsize{$\begin{pmatrix}
1 & 0 & 0 \\
0 & 0 & 1 \\
0 & 1 & 0
\end{pmatrix}$} &  \scriptsize{$\begin{pmatrix}
-1 & 0 & 0 \\
0 & 1 & 0 \\
0 & 0 & -1
\end{pmatrix}$} \rule{0pt}{16pt} 
\end{tabular}
\end{center}
\end{table}

Similarly, the representation $\theta\otimes\rho:S_4 \rightarrow GL(\mathbb{C}^{24})$ decomposes as 
\begin{align}
  \theta\otimes\rho&=2(\theta\otimes\rho)_1\oplus 2(\theta\otimes\rho)_3\oplus 4(\theta\otimes\rho)_4\oplus2(\theta\otimes\rho)_5.
\end{align}
Generating the change of basis
\begin{align}
\label{Matrix_Q_TetrahedronIII}
 Q = \left(\begin{array}{cccccccccccccccccccccccc}
        c_1 & \textbf{0} & c_2 & \textbf{0} & c_3 & \textbf{0} & c_4 & c_ 5 & \textbf{0} & \textbf{0} & c_6 & c_ 7 & \textbf{0} & \textbf{0} & c_8 & c_ 9 & \textbf{0} & \textbf{0} & c_{10} & \textbf{0} & c_{11} & \textbf{0} & c_{12} &  \textbf{0} \\ \hline
         \textbf{0} &c_1 & \textbf{0} & c_2 & \textbf{0} & c_3 & \textbf{0} & \textbf{0} & c_4 & c_ 5 & \textbf{0} & \textbf{0} & c_6 & c_ 7 & \textbf{0} & \textbf{0} & c_8 & c_ 9 & \textbf{0} &   c_{10} & \textbf{0} & c_{11} & \textbf{0} & c_{12} 
        \end{array}\right).
\end{align}
Where $c_i$ be the i-th column of $\tilde{Q}$ in \ref{Matrix_Q_TetrahedronII_rascunho}, and $\textbf{0}$ be the zero column. With
\begin{align}
     \label{Matrix_Q_TetrahedronII_rascunho}
\Tilde{Q}=\left(\begin{array}{rrrrrrrrrrrr}
2    & 1                      &   12                       & 1  & 0 &   0  & 16 &   8  & 0  &   1  &   0  &   -8   \\
2    & -z                     &   w                        & 1  & 0 &   8  & 0  &   0  & 16 &   -1 &   8  &   0    \\
2    & -\mathit{\overline{z}} &   \mathit{\overline{w}}    & 0  & 2 &   8  & 0  &   8  & 0  &   0  &   -8 &   8    \\
2    & 1                      &   12                       & -1 & 0 &   0  & 16 &   -8 & 0  &   -1 &   0  &   8    \\
-2   & z                      &   -w                       & 1  & 0 &   -8 & 0  &   0  & 16 &   -1 &   -8 &   0    \\
-2   & \mathit{\overline{z}}  &   -\mathit{\overline{w}}   & 0  & 2 &   -8 & 0  &   8  & 0  &   0  &   8  &   8    \\
-2   &  -1                    &   -12                      & 1  & 0 &   0  & 16 &   -8 & 0  &   1  &   0  &   8    \\
2    & -z                     &   w                        & -1 & 0 &   -8 & 0  &   0  & 16 &   1  &   -8 &   0    \\
-2   & \mathit{\overline{z}}  &   -\mathit{\overline{w}}   & 0  & 2 &   8  & 0  &   -8 & 0  &   0  &   -8 &   -8   \\
-2   &  -1                    &   -12                      & -1 & 0 &   0  & 16 &   8  & 0  &   -1 &   0  &   -8   \\
-2   & z                      &   -w                       & -1 & 0 &   8  & 0  &   0  & 16 &   1  &   8  &   0    \\
2    & -\mathit{\overline{z}} &   \mathit{\overline{w}}    & 0  & 2 &   -8 & 0  &   -8 & 0  &   0  &   8  &   -8   
\end{array}\right),
 \end{align}
and  $z = \frac{1}{2}+\frac{\sqrt{3}}{2} i$ e $w = 6+6\sqrt{3}i$.
 Although $Q^{-1}$ is not real, $Q^{-1}\cdot S\cdot P$ is real. To see this
let $Q^{-1}=A+ iB$, with $A$ and $B$ are real matrices, and verify that $B\cdot S\cdot P$ is the null matrix. 
We describe the structure of  $S$ after the change of basis, considering the results of Lemmas \ref{lemma_block_zero} and \ref{lema_decompoe_mais}.
\begin{itemize}
    \item[$\bullet$]$\theta_1$ is equivalent to $(\theta\otimes\rho)_1$ with degree $1,$ so we have a block with just one copy.  
    Both subrepresentations have multiplicity equal to $2$, 
therefore the block is $2\times 2$. We will refer to this block as $T_1$(after tetrahedron) and analyzing in Subsection \ref{subsec:block1_tetahedron}.
    \item[$\bullet$] $\theta_4$ is equivalent to $(\theta\otimes\rho)_4$ with degree $3$, therefore, generating three identical blocks $\mathfrak{t}_{44}$.
    The size of the blocks is
    $4\times 2$  given by the multiplicities. We will refer to this block $\mathfrak{t}_{44}$ as $\mathfrak{t}_4$ and analyze it in Subsection \ref{subsec:block4_tetahedron}.
\end{itemize}
As we have no other equivalences, the structure of the matrix $S$ after the change of basis is given by
 {\small
 \begin{equation}
    Q^{-1}SP=\left( \begin{array}{c|c|c|c}
T_1 &  & & \\\hline
    &  & & \\\hline
& \mathfrak{t}_4 &  &  \\\hline
& & \mathfrak{t}_4 &  \\\hline
& &  & \mathfrak{t}_4 \\\hline
   &  & & \end{array}\right).
 \end{equation}
}
The position of blocks is predicted by Lemma \ref{lemma_block_zero}.
The columns of $P$ form the new basis to express the mass vector.
The solutions for the blocks' equations give us the mass vector coordinates. Block $T_1$ determines the coefficients for the first and second vectors. Block $\mathfrak{t}_4$(the three copies) determines the coefficients for columns 3 to 8.
\subsubsection{Prerequisites for block analysis-Reparametrization}
\label{Prerequisites for block analysis-Reparametrization}
The decomposition of an equivariant map reduces the system into subsystems of smaller size. However, in our case, it is convenient to  use more tools from algebra to study these subsystems.

A difficulty appearing in many dynamic system problems is studying expressions involving square roots. A way to perform the calculations is  to find a  convenient substitution to transform the expression into a rational function. This approach was already used  in \cite{Leandro_Finiteness_Bifurcation} and \cite{RAMIREZ_GASULL_Llibre}. Another work discussing some techniques for this approach is a paper by Gasull, Lázaro, and Torregrosa(\cite{Gasull_LAzaro_Torregrosa}). 

In our case, many expressions have denominators
of the form $|q_k-tq_l|=\sqrt{|q_k|^2-2t\langle q_k,q_l\rangle+t^2|q_l|^2}$.
The quadratic polynomial inside the radical has a non-positive discriminant $\Delta$.
 In the case of $\Delta=0$, the expression is already a polynomial. Otherwise we obtained rational expressions by performing a reparameterization 
inspired by the work of Leandro \cite{Leandro_Finiteness_Bifurcation}.
\begin{lemma}
\label{parametrizando_uma_expressao}
Consider an expression
$
\label{expressao_expoente_3_sobre_2}
  r(t)=\frac{n(t)}{(at^2+bt+c)^\frac{3}{2}},$
where $n(t)$ is a rational function and $at^2+bt+c$ is a quadratic polynomial of negative discriminant $\Delta$. Let $(c(u),s(u))= \Big(\frac{1}{2}(u+\frac{1}{u}),\frac{1}{2}(u-\frac{1}{u})\Big),
$ 
and
{\small
\begin{align}
\label{building_blocks}
\kappa(u)=\frac{\sqrt{-\Delta}}{2a}s(u)-\frac{b}{2a}.
 \end{align} }
 Then
 $
r(\kappa(u))=\frac{n(\kappa(u))}{(\frac{-\Delta}{4a})^\frac{3}{2}(c(u))^3}  
$
 is  a rational expression.   Furthermore, the map $\kappa:\Big(\frac{b+\sqrt{b^2-\Delta}}{\sqrt{-\Delta}},\frac{b+2a+\sqrt{(b+2a)^2-\Delta}}{\sqrt{-\Delta}}\Big)\to (0,1)$ is a bijection.
 
 \end{lemma}
\begin{proof}
 The identity  for $r(\kappa(u))$ is a straightforward computation. Moreover, it is easy to see that the expression indeed is rational. As $s:(0,+\infty)\to\mathbb{R}$ is a bijection, then $\kappa$ is also a bijection.  We obtain the desired result by restricting the domain of $\kappa$ to the preimage of $(0,1)$.
\end{proof}

After using the parameterization, we obtain a one-variable polynomial, this allows us to use classical results, as listed below. 
\begin{dfn}
The number of variations of signs in the sequence of real numbers $a_{0}, a_1, \ldots, a_n$  is the number of pairs $(i,j)$ with $i<j$, $a_ia_j<0$, and $a_{i+1}=\cdots = a_{j-1}=0$.
\end{dfn}

\begin{dfn}
Let $p(x)$ be a  real polynomial of degree $n$. Denote $p_0=p(x)$ and by $p_1$ its derivative. The  Sturm sequence of $p(x)$ is the sequence polynomial $p_0, p_1, \ldots, p_n$, where each $p_{i+2}$ is the negative of the remainder of the polynomial division of $p_i$ by $p_{i+1}$.
\end{dfn}
Let $\alpha$ be any real number and $p_0, p_1, \ldots, p_n$ the sequence of Sturm of $p(x)$, we will denote by $v(\alpha)$ the number of variations of signs to $\alpha$ in the  sequence of real numbers $p_0(\alpha), p_1(\alpha), \ldots, p_n(\alpha)$.
\begin{theorem}[Sturm's Theorem]
\label{Sturms_theorem}
Let $p(x)$ be a  real polynomial of degree $n$ with simple roots. Supposing that neither $a$ nor $b$ is a root of the $p(x)$, then the number of roots of $p(x)$ in the interval $a< x< b$ is equal to the difference 
\begin{align}\label{number_roots}
  v(a)-v(b).
\end{align}
\end{theorem}
\begin{proof}
See \cite{Uspensky}.
\end{proof}

Sturm's theorem provides a robust and straightforward method for determining the number of roots in a given interval. Although it is a technique with high computational cost compared to other methods, such as Vincent-Collins-Akritas (VCA) \cite{VCA}, it is sufficiently effective for the purpose of this subsection. Sturm's theorem can be applied to polynomials with multiple roots. In this case, it provides the number of distinct roots but does not determine the multiplicity of the roots \cite{JMT}. We are particularly interested in checking if a polynomial in one variable $p(x)$ has any root in a given interval  $a<x <b$, i.e., we want to investigate whether the difference \eqref{number_roots} is null. Therefore, the condition for $p(x)$ to have simple roots is unnecessary for our investigation.

\subsection{Analyzing the blocks for the Tetrahedron}
\label{subsec: Analyzing the blocks for the Tetrahedron}
\subsubsection{The first block $T_1$}
\label{subsec:block1_tetahedron}
The first column of $P$ in \eqref{Matrix_p_tetraedro} only contributes to the masses associated with the outer tetrahedron because its last four entries are zero. Similarly, the second column contributes only to the masses associated with the inner tetrahedron.
As the coordinates of the vectors are equal, this block indicates if equal masses in each tetrahedron generate a central configuration.
We conclude that there is a central configuration with equal masses in each polyhedron if, and only if, the system $T_1\vec{\mu}=\textbf{0}$ has a nontrivial solution $\vec{\mu}=(\mu_1,\mu_2)$. Moreover, the restrictions are $\mu_1>0$ and $\mu_2>0$, and $\frac{\mu_1}{\mu_2}$ corresponds to the ratio between the outer and inner mass. 

The block $T_1$ is given by
{\small
\begin{equation}
   \label{block1_tetahedron} T_1=\left(\begin{array}{lr}
12 \, c - \frac{3}{8} \, \sqrt{2} & 12 \, c - \frac{3 \, {\left(t + 3\right)}}{{\left(3 \, t^{2} + 2 \, t + 3\right)}^{\frac{3}{2}}} - \frac{\sqrt{3}}{3 \, {\left(t - 1\right)}^{2}}\\
12 \, c t - \frac{3 \, {\left(3 \, t + 1\right)}}{{\left(3 \, t^{2} + 2 \, t + 3\right)}^{\frac{3}{2}}} + \frac{\sqrt{3} t - \sqrt{3}}{3 \, {\left(t - 1\right)}^{3}}  & 12 \, c t - \frac{\frac{3}{8} \, \sqrt{2}}{t^{2}}
\end{array}\right).
\end{equation}
}
It is easy to verify that the determinant of $T_1$ has the form
\begin{equation}
    \det(T_1)=\alpha_1(t)c+\alpha_0(t). \label{determinant_A_tetrahedron}
\end{equation}
Therefore there exists a nontrivial solution for the kernel of $T_1$ if, and only if
$c=c(t):=-\frac{\alpha_0(t)}{\alpha_1(t)}$ or both $\alpha_1(t)$ and $\alpha_0(t)$ are zero.
The next lemma will be helpful in the discussion that follows. To fix the notation, let $(T_1)_{ij}$ be the entries of $T_1$.

\begin{lemma}
\label{lema_sinais_tetrahedro}
\begin{itemize}
    \item[i)] $\alpha_1(t)$ is negative for $t$ in $(0,1)$.
    \item[ii)] $\alpha_0(t)$ is positive for $t$ in $(0,1)$.
    \item[iii)] $c(t):=-\frac{\alpha_0(t)}{\alpha_1(t)}$ is positive for $t$ in $(0,1)$. 
     \item[iv)] $(T_1)_{11}(c(t),t)$ is positive for  $t$ in $(0,1)$.
      \item[v)]There exits a $\delta<1$ such $(T_1)_{12}(c(t),t)$ is negative for $t$ in $(0,\delta)$ and is positive for $t$ in $(\delta,1)$ .
     \end{itemize}
\end{lemma}

\begin{proof}
Equation \eqref{determinant_A_tetrahedron} is given by the expressions
\begin{align}
&\resizebox{.6\hsize}{!}{$
\alpha_1(t)
=-\frac{1}{2} \, {\left(9 \sqrt{2} t - \frac{72 \, {\left(t^{2} + 6 \, t + 1\right)}}{{\left(3 \, t^{2} + 2 \, t + 3\right)}^{\frac{3}{2}}} - \frac{8 \, \sqrt{3}}{t - 1} + \frac{9 \, \sqrt{2}}{t^{2}}\right),} $}\label{alphaumtetraedro} \\
 &\resizebox{.8\hsize}{!}{$\alpha_0(t)=- \frac{9 \, {\left(3 \, t + 1\right)} {\left(t + 3\right)}}{{\left(3 \, t^{2} + 2 \, t + 3\right)}^{3}} - \frac{2 \, \sqrt{3}}{{\left(3 \, t^{2} + 2 \, t + 3\right)}^{\frac{3}{2}} {\left(t - 1\right)}} + \frac{9}{32 \, t^{2}} + \frac{1}{3 \, {\left(t - 1\right)}^{4}}$}.\label{alphazerotetraedro}
\end{align}
Lemma \ref{parametrizando_uma_expressao} applied to denominators $(3t^2+2t+3)^{\frac{3}{2}}$ will transform these expressions into a quotient between two polynomials on $u$. The interval for $u$ in $(\sqrt{2},\sqrt{2}+\sqrt{3})$ corresponds to $t$ in $(0,1)$.
\begin{itemize}
\item[i)] After the change of variables, the polynomial on the numerator of $\alpha_1(t)$ is given by
{\small \begin{align}\label{polinomio_numerador_alpha1_tetraedro} 
    \begin{split}
      &-6\sqrt{3}\sqrt{2} u^{14} + 42  \sqrt{3} u^{13} - 3  {\left(13  \sqrt{3} \sqrt{2} - 48\right)} u^{12} - 6 
         {\left(7  \sqrt{3} - 48  \sqrt{2}\right)}u^{11} \\
        &+ 3  {\left(43  \sqrt{3} \sqrt{2} - 1224\right)} u^{10}- 30  u^{9} {\left(7  \sqrt{3} - 60  \sqrt{2}\right)} + 36  {\left(15  \sqrt{3} \sqrt{2} + 236\right)} u^{8} \\
         &- 4320  \sqrt{2} u^{7} + 36  {\left(15  \sqrt{3} \sqrt{2} - 238\right)} u^{6} + 6  u^{5} {\left(35  \sqrt{3} + 348  \sqrt{2}\right)} \\
         &+3  {\left(43  \sqrt{3} \sqrt{2} + 1296\right)} u^{4} + 6  u^{3} {\left(7  \sqrt{3} + 72  \sqrt{2}\right)} - 39  \sqrt{3} \sqrt{2} u^{2} - 42  \sqrt{3} u - 6  \sqrt{3} \sqrt{2}
        \end{split}
   \end{align}
  }
We apply Sturm's theorem to guarantee that the polynomial \eqref{polinomio_numerador_alpha1_tetraedro} has no root on the interval $(\sqrt{2},\sqrt{2}+\sqrt{3})$. Consequently, $\alpha_1(t)$ has no zeros for $t$ in $(0,1)$.
    Moreover $\displaystyle\lim_{t \to 0^+}\alpha_1(t)=\displaystyle\lim_{t \to 1^-}\alpha_1(t)=-\infty$. We conclude that $\alpha_1(t)$ is negative for $t$ in $(0,1)$.
  \item[ii)] The change of variables transforms $\alpha_0(t)$ into a quotient where the polynomial on the numerator is
 {\small   
    \begin{align}
    \label{polinomio_numerador_alpha0_tetraedro} 
    \begin{split}
        &-81u^{24} +810\sqrt{2} u^{23} -4833u^{22} -972 \sqrt{2}u^{21} +9477  u^{20} + 256770\sqrt{2} u^{19} -1733643u^{18}\\ &+ 1413936  \sqrt{2} u^{17} + 3448278 u^{16} - 5534892\sqrt{2} u^{15} - 3077514 u^{14} + 8276472\sqrt{2}u^{13}\\ & + 2820906 u^{12} - 5711148  \sqrt{2} u^{11} - 3340278 u^{10} + 1548720\sqrt{2} u^{9} + 1995435u^{8}\\& + 426114\sqrt{2} u^{7} + 96795  u^{6} + 14580\sqrt{2} u^{5} + 6561u^{4} + 810  \sqrt{2} u^{3} + 81  u^{2}.
        \end{split}
    \end{align}
    }
 By Sturm's theorem, this polynomial has only one root in the interval $(\sqrt{2},4)$. This root is $\sqrt{2}+\sqrt{3}$. Thus there is no root between $\sqrt{2}$ and $\sqrt{2}+\sqrt{3}$. Hence $\alpha_0(t)$ is different from zero for $t$ in $(0,1)$. Furthermore, $\displaystyle\lim_{t \to 0^+}\alpha_0(t)=\displaystyle\lim_{t \to 1^-}\alpha_0(t)=+\infty$. This proves item ii).
    \item[iii)] Follows from i) and ii).
    \end{itemize}
Items iv) and v) are proved in a similar way. We will omit the proofs.
\end{proof}
By Lemma \ref{lema_sinais_tetrahedro} there exists $\vec{\mu}=(\mu_1,\mu_2)$ a nontrivial solution of $T_1\vec{\mu}=\textbf{0}$ for each $t \in (0,1)$. 
 The equation of the first row is
\begin{equation}
\mu_1(T_1)_{11}+\mu_2(T_1)_{12}=0. \label{relacao_massas_tetraedro}
\end{equation}
 By Lemma \ref{lema_sinais_tetrahedro},  item iv),  $(T_1)_{11}\neq 0$
 ,
then $\mu_2\neq 0$.
So the equation 
\eqref{relacao_massas_tetraedro} is equivalent to
$
\frac{\mu_1}{\mu_2}=-\frac{(T_1)_{12}}{(T_1)_{11}}. $
Substituting the expression for the entries of $T_1$ given in \eqref{block1_tetahedron}, we obtain formula \eqref{formula_razao_massas_tetraedro}, with  $c(t)=-\frac{\alpha_0(t)}{\alpha_1(t)}$. 

Conversely suppose the equation  \eqref{formula_razao_massas_tetraedro} is true and  $c=c(t)=-\frac{\alpha_0(t)}{\alpha_1(t)}$. 
Then equation \eqref{relacao_massas_tetraedro} is also true.
As the determinant of $T_1$ is zero, the second equation of $T_1\vec{\mu}$ is a multiple of the first equation and, therefore, is also zero.   
We have proved the next theorem.
\begin{theorem}
\label{Teorema_1_Tetraedro}
Consider a configuration where positions are the vertices of two nested tetrahedra. Moreover, suppose the masses in each tetrahedron are equal. 
Let $t$ be the ratio between the edges of the two tetrahedra. 
The configuration is a central configuration if and only if
the ratio between the outer mass $\mu_1$ and inner mass $\mu_2$ is given by   
\begin{align}
\frac{\mu_1}{\mu_2}=-\frac{36 \, c(t) - \frac{9 \, {\left(t + 3\right)}}{{\left(3 \, t^{2} + 2 \, t + 3\right)}^{\frac{3}{2}}} - \frac{\sqrt{3}}{{\left(t - 1\right)}^{2}}}{3 \, {\left(12 \, c(t) - \frac{3}{8} \, \sqrt{2}\right)}},\label{formula_razao_massas_tetraedro}\end{align}
where
\begin{align}
\label{c_de_t_tetrahedron}
c(t)=\frac{- \frac{9 \, {\left(3 \, t + 1\right)} {\left(t + 3\right)}}{{\left(3 \, t^{2} + 2 \, t + 3\right)}^{3}} - \frac{2 \, \sqrt{3}}{{\left(3 \, t^{2} + 2 \, t + 3\right)}^{\frac{3}{2}} {\left(t - 1\right)}} + \frac{9}{32 \, t^{2}} + \frac{1}{3 \, {\left(t - 1\right)}^{4}}}{-\frac{1}{2} \, {\left(9\sqrt{2} t - \frac{72 \, {\left(t^{2} + 6 \, t + 1\right)}}{{\left(3 \, t^{2} + 2 \, t + 3\right)}^{\frac{3}{2}}} - \frac{8 \, \sqrt{3}}{t - 1} + \frac{9 \, \sqrt{2}}{t^{2}}\right)}}.
\end{align}
 \end{theorem}
Note that the next theorem also follows from Lemma \ref{lema_sinais_tetrahedro}, items iv) and v)  and formula \eqref{formula_razao_massas_tetraedro}.
\begin{theorem}
\label{Teorema_2_Tetraedro}
Let $\mu_1$ be the mass in the outer tetrahedron, and $\mu_2$ be the mass in the inner tetrahedron in a central configuration of two nested tetrahedra. Let $t$ be the ratio between the edges of the tetrahedra. There is an open interval $(0,\delta)$ such that if $t \in (0,\delta)$, the masses $\mu_1$ and $\mu_2$ must have the same sign. The masses have opposite signs if $t\in (\delta, 1)$. That is, the central configuration is possible if and only if the two tetrahedra are not so close.
\end{theorem}
\subsubsection{Block $\mathfrak{t}_4$}
\label{subsec:block4_tetahedron}
Remember, there are three equal blocks to this one.
So, the same coefficient will appear in three columns. The first coordinate of a solution for the kernel of $\mathfrak{t}_4$ is a coefficient for columns $3$, $5$, and $7$; the second coordinate is a coefficient for columns $4$, $6$,  and $8$ of the matrix $P$ in \eqref{Matrix_p_tetraedro}. 

The block $\mathfrak{t}_4$ is given by 
{\small
\begin{align}
\label{block2_tetahedron} 
\mathfrak{t}_4 = \left(\begin{array}{cc}
0 & -\frac{2{\left(t - 1\right)}}{{\left(3 t^{2} + 2  t + 3\right)}^{\frac{3}{2}}} - \frac{2  {\left(\sqrt{3} t - \sqrt{3}\right)}}{9 {\left(t - 1\right)}^{3}} \\
-4  c + \frac{1}{8}  \sqrt{2} & -4  c t + \frac{3  t + 1}{{\left(3  t^{2} + 2  t + 3\right)}^{\frac{3}{2}}} - \frac{\sqrt{3} t - \sqrt{3}}{9  {\left(t - 1\right)}^{3}} \\
\frac{2  {\left(t - 1\right)}}{{\left(3  t^{2} + 2  t + 3\right)}^{\frac{3}{2}}} + \frac{2  {\left(\sqrt{3} t - \sqrt{3}\right)}}{9  {\left(t - 1\right)}^{3}} & 0 \\
-4  c + \frac{t + 3}{{\left(3  t^{2} + 2  t + 3\right)}^{\frac{3}{2}}} + \frac{\sqrt{3} t - \sqrt{3}}{9  {\left(t - 1\right)}^{3}} & -4  c t + \frac{\sqrt{2}}{8  t^{2}}
\end{array}\right).
\end{align}
}
One of the minors $2\times2$ of $\mathfrak{t}_4$ is independent of $c$, given by
\begin{align}\label{p2}
\resizebox{.7\hsize}{!}{$  \frac{4}{81} \, {\left(\frac{9 \, t}{{\left(3 \, t^{2} + 2 \, t + 3\right)}^{\frac{3}{2}}} - \frac{9}{{\left(3 \, t^{2} + 2 \, t + 3\right)}^{\frac{3}{2}}} + \frac{\sqrt{3} t}{{\left(t - 1\right)}^{3}} - \frac{\sqrt{3}}{{\left(t - 1\right)}^{3}}\right)}^{2}.$}
\end{align}
Using Lemma \ref{parametrizando_uma_expressao} will produce an expression with a numerator
\begin{equation}
\begin{aligned}-4 \, u^{8} + 12 \, \sqrt{2} u^{7} - 48 \, u^{6} + 8 \, \sqrt{2} u^{5} + 36 \, u^{4} + 12 \, \sqrt{2} u^{3}. \label{block_2_tetrahedron_polynomial_of_u}
\end{aligned}
\end{equation}
 By Sturm's theorem, the polynomial \eqref{block_2_tetrahedron_polynomial_of_u} has not a root in the interval in $(\sqrt{2},\sqrt{2}+\sqrt{3})$. Hence, the expression \eqref{p2}  has no zeros in the interval $(0,1)$; this proves the next theorem.

\begin{theorem}
Consider a central configuration where positions are the vertices of two nested tetrahedra. The masses in each tetrahedron are equal. 
\end{theorem}
\subsection{ Nested Octahedra}
\label{sec:Octahedron} 
In this subsection, we reobtain the results of \cite{Liu_Tao}. Specifically, subsection \ref{subsec:Block_1_Octahedron} proves Theorem 1.3 of that reference. While subsections \ref{subsec:Block_5_Octahedron} and
\ref{subsec:Block_9_Octahedron} prove Theorem 1.2 of the same reference.

For the outer octahedron, we choose the following positions $q_1=(1,0,0)^T,q_2=(-1,0,0)^T,q_3=(0,1,0)^T,q_4=(0,-1,0)^T,q_5=(0,0,1)^T,q_6=(0,0,-1)^T.$

The group of symmetries of the octahedron is $O_h,$ a group of order $48$. This group has 
ten irreducible subrepresentations, with degrees given explicitly by
\begin{equation}
(\deg(r_1),\ldots,\deg(r_{10}))=(1,1,1,1, 2, 2, 3,3,3,3).
\end{equation}
The representations $\theta:O_h \rightarrow GL(\mathbb{C}^{12})$ and $\theta\otimes\rho:O_h \rightarrow GL(\mathbb{C}^{36})$ decomposes into irreducible
subrepresentations as

\begin{align}
\label{decomposition_of_Octahedron_Representations}
\theta&=2\theta_1\oplus 2\theta_5\oplus 2\theta_9,\\
    \theta\otimes\rho&=2(\theta\otimes\rho)_1\oplus 2(\theta\otimes\rho)_5\oplus2(\theta\otimes\rho)_7\oplus2(\theta\otimes\rho)_8\oplus4(\theta\otimes\rho)_9\oplus2(\theta\otimes\rho)_{10}.\nonumber
\end{align}

Hence, there are three non-null blocks denoted by $O_1, O_5, O_9$. The block $O_1$ from the equivalence between $\theta_1$ and $(\theta\otimes\rho)_1$  is of size $2\times 2$(the multiplicities) e have just one copy(the degree).
Block $O_5$ is of size $4\times4$ divided in two copies $\mathfrakO_5$ of size $2\times 2$.
While block $O_9$ is of size $12\times 6$ divided in three copies $\mathfrakO_9$ of size $4\times 2$. We analyze them below.
\subsubsection{Block $O_1$}
\label{subsec:Block_1_Octahedron}

The block $O_1$ in the octahedron case is given by
{\small
\begin{align}
 \label{block_1_Octahedron_Case}O_1=\left(\begin{array}{cc}
6 \, c - \frac{(4\sqrt{2}+1)}{4} & \,6 \, c - \frac{1}{(t + 1)^2}- \frac{4}{{\left(t^{2} + 1\right)}^{\frac{3}{2}}} - \frac{1}{{\left(t - 1\right)}^{2}} \\
6 \, c t - \frac{1}{(t + 1)^2} - \frac{4 \, t}{{\left(t^{2} + 1\right)}^{\frac{3}{2}}} + \frac{1}{{\left(t - 1\right)}^{2}}& 6 \, c t - \frac{4\sqrt{2}+1}{t^{2}}
\end{array}\right).
\end{align}
}
It is easy to see that 
$\det(O_1)=\beta_1(t)c+\beta_0(t)$. 
We verify a lemma similar to Lemma \ref{lema_sinais_tetrahedro}, and get the same conclusions for this case
as in subsection \ref{subsec:block1_tetahedron}. To be concise,
 we will omit the calculations. Also,
from now on,  we will omit the relationship between the masses and the edges of polyhedra. 
The reader can use the same arguments as before to find.
\subsubsection{Block $\mathfrakO_5$}
\label{subsec:Block_5_Octahedron}
The block $\mathfrakO_5$ is given by
{\small
\begin{align}
  \mathfrakO_5=  \left(\begin{array}{cc}
\frac{\sqrt{2}}{2} - \frac{1}{4} & -\frac{1}{(t + 1)^2}+ \frac{2}{{\left(t^{2} + 1\right)}^{\frac{3}{2}}}  - \frac{1}{{(1 - t) }^{2}} \\
-\frac{1}{(t + 1)^2} + \frac{2 \, t}{{\left(t^{2} + 1\right)}^{\frac{3}{2}}}+ \frac{1}{{( 1 - t )}^{2}} & \frac{1}{ t^{2}}\left(\frac{\sqrt{2}}{2} - \frac{1}{4}\right)
\end{array}\right).
\end{align}}
Using Sturm's Theorem  and Lemma \ref{parametrizando_uma_expressao}, we verify that the determinant is different from zero for $t \in (0,1)$.
\subsubsection{Block $\mathfrakO_9$}
\label{subsec:Block_9_Octahedron}
The block $\mathfrakO_9$ is given by
\begin{equation}
\label{Block_O9_Octahedron_Case}
\mathfrakO_9=\left(\begin{array}{cc}
-2 \, c + \frac{1}{4} & -2 \, c t + \frac{1}{(t + 1)^2}  - \frac{1}{{(1-t)}^{2}} \\
-4 \, c + \sqrt{2} & -4 \, c t + \frac{4 \, t}{{\left(t^{2} + 1\right)}^{\frac{3}{2}}} \\
-2 \, c + \frac{1}{(t + 1)^2} + \frac{1}{{(1- t)}^{2}}  & -2 \, c t + \frac{1}{4t^{2}} \\
-4 \, c + \frac{4}{{\left(t^{2} + 1\right)}^{\frac{3}{2}}} & -4 \, c t + \frac{\sqrt{2}}{t^{2}}
\end{array}\right).
\end{equation}
In this case, we calculate the $2\times2$ minors of $\mathfrakO_9$ and try to show that at least one is differs from zero to $t \in (0,1)$.
We choose the following two minors: 
\begin{align}
\label{p_91_Octahedron}
&\resizebox{.7\hsize}{!}{$p_{91}(c,t)={-\frac{1}{2} \, c {\left(t - \frac{4}{t + 1} - \frac{4}{t - 1} + \frac{1}{t^{2}}\right)} + \frac{1}{16 \, t^{2}} - \frac{1}{{\left(t + 1\right)}^{4}} + \frac{1}{{\left(t - 1\right)}^{4}}}$},\\
&\resizebox{.99\hsize}{!}{$p_{92}(c,t)={-c {\left(t - \frac{8 \, t}{{\left(t^{2} + 1\right)}^{\frac{3}{2}}} + \frac{2 \, \sqrt{2}}{t^{2}} - \frac{4}{{\left(t + 1\right)}^{2}} + \frac{4}{{\left(t - 1\right)}^{2}}\right)}}+ \frac{\sqrt{2}}{4 \, t^{2}} - \frac{4}{{\left(t^{2} + 1\right)}^{\frac{3}{2}} {\left(t + 1\right)}^{2}} + \frac{4}{{\left(t^{2} + 1\right)}^{\frac{3}{2}} {\left(t - 1\right)}^{2}}$}.\nonumber
\end{align}
We solve the first equation in \eqref{p_91_Octahedron} for $c$ substitute in the second. Using Lemma \ref{parametrizando_uma_expressao} and Sturm's
 Theorem, we ensure that if $p_{91}$ has a zero, then 
$p_{92}$ does not have a zero in $(0,1)$;
this ends the discussion of the octahedron case, and we can enunciate the following  theorem.
\begin{theorem}
Consider a central configuration where positions are the vertices of two nested octahedra. The masses in each octahedron are equal. 
\end{theorem}
\section{Applications: The inverse and direct problem for central configurations of two nested cube}
\label{Applications: The inverse and direct problem for central configurations of two nested cubes}

\subsection{Prerequisites for block analysis - Lift} 

Before analyzing the blocks of the cube case, we needed more algebra tools, as it was not possible to use reparameterization to obtain rational expressions. In this
case, we have introduced new variables to lift the curve to a  higher
dimension space. Consequently, Sturm's theorem could not be applied.

\begin{lemma}
\label{lemma_of_lift}
Let $f:(0,1)\to \mathbb{R}$ be given by
$\displaystyle
f(t)=q(t)+\sum_{p=1}^{L}\dfrac{d_p(t)}{(a_pt^2+b_pt+c_p)^{\frac{3}{2}}}
$
where $q$ and $d_p$ are rational functions of $t$. 
Consider $g:(0,1)\to \mathbb{R}^{L}$ be the curve defined by
\begin{align}
      g(t)=(t, u_1(t),u_2(t),\ldots,u_L(t)), \text{ where }
    u_p(t)=\sqrt{a_pt^2+b_pt+c_p}.   \label{curve_of_lift}
\end{align}
 
 Let $I_p$ be the image of the interval $(0,1)$ by $u_p$.
There exists a rational function $F$ defined on the $L+1$-dimensional block 
$\mathcal{B}=\displaystyle(0,1)\times\prod_{p=1}^{L} I_p$
such that $g$ is a lift of $f$ to $\mathcal{B}$.
\end{lemma}
\begin{proof}
Let
  \begin{equation}
 \label{Function_of_lift}
F(t,u_1,\ldots,u_L)=\sum_{p=1}^{L} \frac{d_p(t)}{u_p^3}+q(t).
 \end{equation}

Then function $F$  is rational and $F\circ g(t) = f(t)$.
\end{proof}

Using the notation of previous lemma.
It is easy to see that a zero of $f$ in $(0,1)$ will produce a zero of $F$ in $\mathcal{B}$. 
Our goal is to show that this does not occur. To determine when a multivariate polynomial does not have zeros on a box, we refer the reader to the work of Barros and Leandro  \cite{Barros_Leandro}, which discusses extensions of Vincent's Theorem (also see \cite{Alesina_Galuzzi1998,Uspensky} and references therein).

Let $p(x)$ be a real polynomial of one variable, and $I=(a,b]$ an interval, define
\begin{align}
\label{Polinomio_Uma_Variavel_Restrito}
    p|_{I}(u)=(u+1)^{\text{degree $p$}}\cdot p\left(\frac{au+b}{u+1}\right).
\end{align}
Then, $p|_{I}$ is the numerator of the composition $p\left(\frac{au+b}{u+1}\right)$.
Given the multivariate real polynomial $p(u_1,u_2,\ldots,u_n)$, we can pick a variable $u_j$ and consider it a polynomial  whose coefficients depend on the other variables. Let us denote this  choice by $p_{\hat{u}_j}(u_1,u_2,\ldots,u_n)$.
Given  a box $\mathcal{B}=\prod_{k=1}^{n} I_k$, we can iterate in a natural way the application given in \eqref{Polinomio_Uma_Variavel_Restrito} and define
\begin{align}
\label{Polinomio_n_Variaveis_Restrito}    p|_{\mathcal{B}}=(\ldots (p_{\hat{u}_n}|_{I_n})\ldots)|_{I_1}.
\end{align}
\begin{theorem}(Barros-Leandro)
\label{Theorem_Barros_Leandro}
The real polynomial $p(u_1,u_2,\ldots,u_n)$ does not have zeros on the box $\mathcal{B}=\prod_{k=1}^{n}[a_k,b_k]$ if and only if there exists a finite covering of $\mathcal{B}\subset \cup \mathcal{B}_{\alpha}$ such that $p|_{\mathcal{B}_{\alpha}}$ is a polynomial whose all coefficients have the same sign.
\end{theorem}

The polynomial $p|_{I}$ is obtained by substituting the variable with a linear fractional M\"{o}bius transformation $x = \frac{au+b}{u+1}$. We will refer to the coefficients of $p|_{I}$(or $p|_{\mathcal{B}}$) as M\"{o}bius coefficients.

In our case, a partition of the interval $(0,1)$ will determine a cover by blocks of the lifted curve. When necessary, we will obtain a partition by successive bisections of the interval $(0,1)$. As an illustration, in the following example, we present the method employed.
\begin{example}
\label{example_1_of_patition}
Consider the curve given by $g(t) = (g_1(t),g_2(t)) = (t, (t-t_0)^2+a)$ and the polynomial $P(u,v) = \left(v^2-b^2u^2\right)^3\cdot\left((u-t_0)^2+(v-a-\epsilon)^2\right)^3$,
where $\epsilon$ is a positive parameter. Provided $a>b>0$, it is easy to see that $\alpha(t) =  P\circ g(t) $ does not have zeros for $t \in (0,1)$.
In fact, $g_2(t)\geq a>b>bt = bg_1(t)>0$ for   $t \in (0,1)$, so 
$((g_2(t))^2-b^2(g_1(t))^2)^3\neq 0$. The second factor of $P(u,v)$ is zero if and only if 
$(u,v) = (t_0,a+\epsilon)$ which is not on the trace of the curve $g$. Although the curve has the point $(t_0,a)$ in its trace.
Depending on the parameter chosen, it is necessary to have more than one block to cover the curve $\alpha$ without containing a zero of $P$. 
To illustrate, let us choose the parameters
$a = \frac{1}{4}, b=\frac{1}{7}, t_0=\frac{7}{10}, \epsilon = \frac{1}{5}.$
For an interval $I$ let us use the notation $\mathcal{B}_I = g_1(I)\times g_2(I)$. We can see in Figure \ref{fig: figure_Partition_example} that the block $\mathcal{B}_{[0,1]}$ contains zeros of $P$. While $\mathcal{B}_{[0,\frac{1}{4}]}\cup\mathcal{B}_{[\frac{1}{4},\frac{1}{2}]}\cup \mathcal{B}_{[\frac{1}{2}, \frac{3}{4}]}\cup \mathcal{B}_{[ \frac{3}{4}, 1]} $ determines a covering by blocks where $P$(and therefore $\alpha$) does not have zeros.
\begin{figure}[H]
\centering
\begin{tikzpicture}[scale=4.5]
    \def\tzero{0.7}
    \def\a{0.25}
    \def\epsilon{0.2}

    \draw[->] (-0.1,0) -- (1.1,0) node[anchor=north] {$g_1(t)$};
    \draw[->] (0,-0.1) -- (0,1.2) node[anchor=east] {$g_2(t)$};

    \fill[red!20] (0,{(0-\tzero)^2 + \a}) rectangle (0.25,{(0.25 -\tzero)^2 + \a});

    \fill[yellow!20] (0.25,{(0.25-\tzero)^2 + \a}) rectangle (0.5,{(0.5 -\tzero)^2 + \a});
 
    \fill[blue!20] (0.5,{(0.5 - \tzero)^2 + \a}) rectangle (0.75 ,{(0.75 -\tzero)^2 + \a});

     \fill[green!20] (0.75,{(0.75 - \tzero)^2 + \a+0.1}) rectangle (1 ,{(0.75 -\tzero)^2 + \a});

    \draw[thick, blue, domain=0:1, samples=100] plot (\x,{(\x-\tzero)^2 + \a});

    \draw[dotted, step=0.25] (0,0) grid (1, 1);

    \fill (\tzero,  \a + \epsilon) circle (0.3pt) node[anchor=south] {\scriptsize {$(t_0, a + \varepsilon)$}};
    
     \node[below] at (0.25,0) {$\tfrac{1}{4}$};
     \node[below] at (0.5,0) {$\tfrac{1}{2}$};
     \node[below] at (0.75,0) {$\tfrac{3}{4}$};
    \node[below] at (1,0) {1};

      \node[left] at (0,1) {1};
     \node[left] at (0,\a) {$\tfrac{1}{5}$};

\end{tikzpicture}
\caption{Curve $g(t) = (t, (t - t_0)^2 + a)$  with the blocks $B_{[0, \frac{1}{4}]}$,$B_{[ \frac{1}{4}, \frac{1}{2}]}$, $B_{[\frac{1}{2}, \frac{3}{4}]}$, and $B_{[\frac{3}{4}, 1]}$ are shown in red, yellow, blue, and green, respectively. The values of the parameters are $a = \frac{1}{4}, b=\frac{1}{7}, t_0=\frac{7}{10}, \epsilon = \frac{1}{5}$.}
\label{fig: figure_Partition_example}
\end{figure}
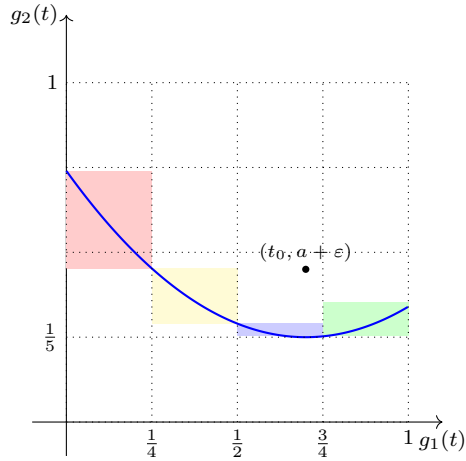

For the values of parameters $a = \frac{1}{4}, b=\frac{1}{7}, t_0=\frac{7}{10}, \epsilon = \frac{1}{5}$, the polynomial is given by
\begin{align*}
    P(u,v) = -\frac{1}{7529536000000} \, {\left(4 \, {\left(10 \, u - 7\right)}^{2} + {\left(20 \, v - 9\right)}^{2}\right)}^{3} {\left(u^{2} - 49 \, v^{2}\right)}^{3}.
\end{align*}
Let us show how to use the method to prove that in the covering 
$\mathcal{B}_{[0,\frac{1}{4}]}\cup\mathcal{B}_{[\frac{1}{4},\frac{1}{2}]}\cup \mathcal{B}_{[\frac{1}{2}, \frac{3}{4}]}\cup \mathcal{B}_{[ \frac{3}{4}, 1]} $
this polynomial has no roots. 
Explicitly, we have
$\mathcal{B}_{[0,\frac{1}{4}]} = [0,\frac{1}{4}]\times [\frac{181}{400},\frac{37}{50}]$ which induces the change of variables by M\"{o}bius transformations
$u =\dfrac{0x+\frac{1}{4}}{x+1}$ and $v =\dfrac{\frac{181}{400}y+\frac{37}{50}}{y+1}$. Taking the numerator of the resulting rational expression we obtain
\begin{align*}
    & \Bigg( \begin{aligned}
        & 1605289 \, x^{2} y^{2} + 5250448 \, x^{2} y + 3210578 \, x y^{2} + 4293184 \, x^{2} \\
        & + 10500896 \, x y + 1565289 \, y^{2} + 8586368 \, x + 5170448 \, y + 4253184
    \end{aligned} \Bigg)^{3} \times \\
    & \Bigg( \begin{aligned}
        & 78401 \, x^{2} y^{2} + 157032 \, x^{2} y + 44802 \, x y^{2} + 91856 \, x^{2} \\
        & + 90064 \, x y + 6401 \, y^{2} + 71712 \, x + 13032 \, y + 19856
    \end{aligned} \Bigg)^{3},
\end{align*}
a polynomial where all coefficients are positive. 
By Theorem \ref{Theorem_Barros_Leandro}, $P(u,v)$ has no roots in $\mathcal{B}_{[0,\frac{1}{4}]}$.
For the block $\mathcal{B}_{[\frac{1}{4}, \frac{1}{2}]}$ the resulting polynomial is 
\begin{align*}
    & \Bigg( \begin{aligned}
        & 649344 \, x^{2} y^{2} + 2037608 \, x^{2} y + 1278688 \, x y^{2} + 1595289 \, x^{2} \\
        & + 4035216 \, x y + 619344 \, y^{2} + 3170578 \, x + 1977608 \, y + 1565289
    \end{aligned} \Bigg)^{3} \times \\
    & \Bigg( \begin{aligned}
        & 36496 \, x^{2} y^{2} + 64672 \, x^{2} y + 36992 \, x y^{2} + 32401 \, x^{2} \\
        & + 57344 \, x y + 10496 \, y^{2} + 28802 \, x + 12672 \, y + 6401
    \end{aligned} \Bigg)^{3},
\end{align*}

For the block $\mathcal{B}_{[\frac{1}{2}, \frac{3}{4}]}$ the resulting polynomial is 

\begin{align*}
    & \Bigg( \begin{aligned}
        & 28125x^{2}y^{2} + 66050x^{2}y + 53750xy^{2} + 38709x^{2} \\
        & + 127100xy + 25000y^{2} + 74918x + 59800y + 35584
    \end{aligned} \Bigg)^{3} \times \\
    & \Bigg( \begin{aligned}
        & 800x^{2}y^{2} + 1440x^{2}y + 600xy^{2} + 656x^{2} \\
        & + 880xy + 425y^{2} + 312x + 690y + 281
    \end{aligned} \Bigg)^{3}.
\end{align*}

And for the block $\mathcal{B}_{[ \frac{3}{4}, 1]}$ the resulting polynomial is

\begin{align*}
    & \Bigg( \begin{aligned}
        & 409849x^{2}y^{2} + 1166128x^{2}y + 759698xy^{2} + 816304x^{2} \\
        & + 2212256xy + 339849y^{2} + 1572608x + 1026128y + 746304
    \end{aligned} \Bigg)^{3} \times \\
    & \Bigg( \begin{aligned}
        & 6641x^{2}y^{2} + 7752x^{2}y + 17282xy^{2} + 2336x^{2} \\
        & + 23504xy + 20641y^{2} + 8672x + 35752y + 16336
    \end{aligned} \Bigg)^{3}.
\end{align*}
As in each block the resulting polynomial has all coefficients with same sign, then $P(u,v)$ has no roots in the covering.

The condition about same sign of coefficients, is sufficient but not necessary. Indeed, one can verify that $\mathcal{B}_{[0,\frac{1}{2}]}\cup \mathcal{B}_{[\frac{1}{2}, \frac{3}{4}]}\cup \mathcal{B}_{[ \frac{3}{4}, 1]} $ is a covering without roots of $P(u,v)$. Nevertheless the resulting polynomial associated to the block $\mathcal{B}_{[0,\frac{1}{2}]}$ has coefficients with positive and negative signs.
\end{example}
To perform the calculations, one should be able to analyze the sign of the M\"{o}bius coefficients at \eqref{Polinomio_Uma_Variavel_Restrito} or \eqref{Polinomio_n_Variaveis_Restrito}. These coefficients can be obtained by looking for the coefficient of each power. 
Explicitly let
$ p(x) =\displaystyle\sum_{i=0}^{D} c_ix^i,$
then 
\begin{equation}
\label{formulae_for_coefficients_ove_variable}
 p_{I}(u)=\sum_{k=0}^{D} \displaystyle\alpha_k u^k, \text{ with }  \alpha_k=\sum_{i=0}^{D}c_i\Big(\sum_{p=0}^i\binom{i}{p}\binom{D-i}{k-p}a^{p}\cdot b^{i-p}\Big). 
\end{equation}
This formula generalizes in a natural way for more variables. For example, for two variables, if 
\begin{align}
&p(x,y) =\displaystyle\sum_{i=0}^{D_x}\sum_{j=0}^{D_y} c_{ij}x^iy^j,\text{ then}\nonumber\\ 
&p|_{[a,b]\times[c,d]}(u,v)=\sum_{k=0}^{D_x}\sum_{l=0}^{D_y} \alpha_{kl}u^kv^l, \text{ with } \nonumber\\ 
&\label{formulae_for_coefficients}\alpha_{kl}=\sum_{i=0}^{D_x}\sum_{j=0}^{D_y}c_{ij}\Big(\sum_{s=0}^i\binom{i}{s}\binom{D_x-i}{k-s}a^{s}\cdot b^{i-s}\Big)\cdot\Big(\sum_{r=0}^j\binom{j}{r}\binom{D_y-j}{l-r}c^{r}\cdot d^{j-r}\Big). 
\end{align}
These formulas allow the calculation of each coefficient independently, 
thus using less memory or even partitioning the calculation where each step is given by the degree. We used this fact in our calculations because it requires fewer computational resources than directly substituting expressions. 


\subsection{Nested Cube}
The symmetry group of the cube is the same as the octahedron, since they are dual to each other. Hence, we refer to the group structure of the previous section(see Section \ref{sec:Octahedron}
).
For the outer cube, we choose the positions {\small$q_1=(1,1,1)^T$}, {\small$q_2=(1,1,-1)^T$}, {\small$q_3 = (1,-1,1)^T,$}
{\small$q_4 = (-1,1,1)^T$}, {\small$q_5=(1,-1,-1)^T$}, {\small$q_6 = (-1,1,-1)^T$}, {\small$q_7=(-1,-1,1)^T$}, {\small$q_8=(-1,-1,-1)^T$}. 
The representations are decomposed as 

\begin{align*}
    \theta &= 2\theta_1 \oplus 2\theta_2 \oplus 2\theta_7 \oplus 2\theta_8, \\
    \theta \otimes \rho &= 
    \begin{aligned}[t]
        &2(\theta\otimes\rho)_1 \oplus 2(\theta\otimes\rho)_2 \oplus 2(\theta\otimes\rho)_5 \oplus 2(\theta\otimes\rho)_6 \\
        &\oplus 4(\theta\otimes\rho)_7 \oplus 4(\theta\otimes\rho)_{8} \oplus 2(\theta\otimes\rho)_9 \oplus 4(\theta\otimes\rho)_{10}
    \end{aligned}
    \end{align*}

Due to the equivalence between subrepresentations of the same indexes, this decomposition induces the following block structure: Two blocks, $\mathcal{C}_1$ and $ \mathcal{C}_2$, of size $2\times 2$ each block.
Two blocks, $\mathfrakC_7$, and $ \mathfrakC_8$, of size $4\times 2$, with three copies each.

\subsubsection{Block $\mathcal{C}_1$}
\label{subsubsec:Block_C1_Cube}
The block $\mathcal{C}_1$ has a determinant given by 
$\gamma_1(t)c+\gamma_0(t)$ where $\gamma_1(t)$ and $\gamma_0(t)$ are given by equations \eqref{Cubo_gamma_1} and \eqref{Cubo_gamma_0}.
We will prove a lemma very similar to Lemma \ref{lema_sinais_tetrahedro}.
\begin{lemma}
\label{lema_sinais_cubo}
\begin{itemize}
    \item[i)] $\gamma_1(t)$ is negative for $t$ in $(0,1)$.
    \item[ii)] $\gamma_0(t)$ is positive for $t$ in $(0,1)$.
    \item[iii)] $c(t):=-\frac{\gamma_0(t)}{\gamma_1(t)}$ is positive for $t$ in $(0,1)$. 
     \item[iv)] $(\mathcal{C}_1)_{11}(c(t),t)$ is positive for $t$ in $(0,1)$.
      \item[v)]There exists a $\delta<1$ such that $(\mathcal{C}_1)_{12}(c(t),t)$ is negative for $t$ in $(0,\delta)$ and is positive for $t$ in $(\delta,1)$ .
     \end{itemize}
\end{lemma}

Note that the following theorems will follow from the Lemma \ref{lema_sinais_cubo}. These theorems are analogous to theorems \ref{Teorema_1_Tetraedro} and
\ref{Teorema_2_Tetraedro}.
\begin{theorem}
\label{Teorema_1_Cubo}
Consider a configuration in which the positions are the vertices of two nested cubes and $t$ is the ratio between the edges of the two cubes. Moreover, suppose that the masses in each cube are equal. 
The configuration is a central configuration if and only if
the ratio between the outer mass $\mu_1$ and inner mass $\mu_2$ is given by   
\begin{align}
\frac{\mu_1}{\mu_2} = -\frac{24 \, c(t) - \frac{\sqrt{3} t + \sqrt{3}}{3 \, {\left(t+ 1\right)^{3}}} - \frac{3 \, {\left(t + 3\right)}}{{\left(3 \, t^{2} + 2 \, t + 3\right)^{\frac{3}{2}} }} + \frac{3 \, {\left(t - 3\right)}}{{\left(3 \, t^{2} - 2 \, t + 3\right)^{\frac{3}{2}} }} - \frac{\sqrt{3} t - \sqrt{3}}{3 \, {\left(t - 1\right)}^{3}}}{24 \, c(t) - \frac{1}{12} \, \sqrt{3} - \frac{3}{8} \, \sqrt{2} - \frac{3}{4}}, \label{razao_massas_cubo}
\end{align}

where $c(t) = \dfrac{-\gamma_0(t)}{\gamma_1(t)}$ with
\begin{align}
\gamma_1(t) = -&\left((2\sqrt{3} + 9\sqrt{2}  + 18 ) t - \frac{72 \, {\left(t^{2} + 6 \, t + 1\right)}}{{\left(3 \, t^{2} + 2 \, t + 3\right)}^{\frac{3}{2}}} + \frac{72 \, {\left(t^{2} - 6 \, t + 1\right)}}{{\left(3 \, t^{2} - 2 \, t + 3\right)}^{\frac{3}{2}}}\right.\nonumber\\
&\left.- \frac{8 \, \sqrt{3}}{t + 1} - \frac{8 \, \sqrt{3}}{t - 1} + \frac{2 \, \sqrt{3} + 9 \, \sqrt{2} + 18}{t^{2}} \right).\label{Cubo_gamma_1}
\end{align}
and

\begin{align}
    \gamma_0(t) &= 
    \begin{aligned}[t]
        &- \frac{9 (3t + 1) (t + 3)}{(3t^{2} + 2 t + 3)^{3}} + \frac{9 (3t - 1) (t - 3)}{(3 t^{2} - 2t + 3)^{3}} \\
        &+ \frac{6\sqrt{3} \sqrt{2} + 12\sqrt{3} + 54\sqrt{2} + 83}{96t^{2}} \\
        &- \frac{4\sqrt{3}}{(3t^{2} + 2t + 3)^{\frac{3}{2}} (t + 1)} - \frac{2\sqrt{3}}{(3t^{2} - 2t + 3)^{\frac{3}{2}} (t + 1)} \\
        &- \frac{2\sqrt{3}}{(3t^{2} + 2t + 3)^{\frac{3}{2}} (t - 1)} - \frac{4\sqrt{3}}{(3t^{2}-2t+3)^{\frac{3}{2}} (t- 1)} \\
        &- \frac{144t}{(3t^{2}+2t+3)^{\frac{3}{2}} (3t^{2}-2t+3)^{\frac{3}{2}}} \\
        &- \frac{1}{3(t+1)^{4}} + \frac{1}{3(t-1)^{4}}
    \end{aligned}
    \label{Cubo_gamma_0}
\end{align}

 \end{theorem}
\begin{theorem}
\label{Teorema_2_Cubo}
Let $\mu_1$ be the mass in the outer cube and $\mu_2$ be the mass in the inner cube in a central configuration of two nested cubes whose ratio between the edges is $t$. There is an open interval $(0,\delta)$ such that if $t \in (0,\delta)$, the masses $\mu_1$ and $\mu_2$ must have the same sign. The masses have opposite signs if $t\in (\delta, 1)$. 
\end{theorem}

We shall proceed to the proof of Lemma \ref{lema_sinais_cubo}.
\begin{proof}
Item {\it i)}.
\end{proof}
Let us first consider $\gamma_1$ in \eqref{Cubo_gamma_1}.
From the last two terms, we see that $ \lim_{t\to 0^+} \gamma_1(t) = \lim_{t\to 1^-} = -\infty$. We only need to prove that this function does not have zeros on the interval $(0,1)$.
We use the Lemma \ref{lemma_of_lift} to get a lift of $\gamma_1$ with
\begin{equation}
\label{lift_for_the_cube}
g(t) =  (t, u_1(t),u_2(t))=(t, \sqrt{3t^2 + 2t + 3}, \sqrt{3t^2 - 2t + 3}),
\end{equation}

 given by
\begin{align*}
\gamma_1(t,u_1,u_2)=&-2 \, \sqrt{3} t - 9 \, \sqrt{2} t - 18 \, t + \frac{8 \, \sqrt{3}}{t + 1} + \frac{8 \, \sqrt{3}}{t - 1} - \frac{2 \, \sqrt{3} + 9 \, \sqrt{2} + 18}{t^{2}} \\&+ \frac{72 \, {\left(t^{2} + 6 \, t + 1\right)}}{u_{1}^{3}} - \frac{72 \, {\left(t^{2} - 6 \, t + 1\right)}}{u_{2}^{3}}. \nonumber
\end{align*}

We look at $\gamma_1(t,u_1,u_2)$ as a rational function and consider the polynomial at the numerator. Following the idea illustrated in Example \ref{example_1_of_patition}, we will prove that this polynomial does not contain zeros on a covering of the image of $g$. 
For an interval $I$, let us use the notation $\mathcal{B}_I = I\times u_1(I)\times u_2(I)$ for the correspondent block that contains the image of $g$.
We verify that the partition $\mathcal{B}_{(0,\frac{1}{2}]}\cup \mathcal{B}_{[\frac{1}{2}, 1)} $ determines a partition where the numerator of $\gamma_1$ does not contain zeros.

To give more details, we describe what occurs with $\mathcal{B}_{(0,\frac{1}{2}]}$.
Using the expression of $g,$ we verify that
$\mathcal{B}_{(0,\frac{1}{2}]} = (0,\frac{1}{2}]\times (\sqrt{3}, \frac{\sqrt{19}}{2}]\times [ \sqrt{\frac{8}{3}}, \sqrt{3}], $ the block obtained by calculating the maximum and minimum of each coordinate of $g(t)$ at the interval $(0, \frac{1}{2}]$.
The substitution \eqref{Polinomio_n_Variaveis_Restrito} is given by
$(t, u_1,u_2)=\left(\frac{0x+\frac{1}{2}x}{x+1},\frac{\sqrt{3}y+\frac{\sqrt{19}}{2}}{y+1},\frac{\sqrt{\frac{8}{3}}z+\sqrt{3}}{z+1}\right).$ Consider the polynomial at the numerator of the resulting rational expression, this polynomial is given by 

\begin{align*}
    & \begin{aligned}[t]
        &36864 \, \sqrt{3} \sqrt{2} \sqrt{\frac{2}{3}} x^{6} y^{3} z^{3} + 73728 \, \sqrt{3} \sqrt{\frac{2}{3}} x^{6} y^{3} z^{3} + \cdots\\
        &\cdots + 12288 \, \sqrt{19} \sqrt{3} \sqrt{\frac{2}{3}} x^{6} y^{2} z^{3} + 175446 \sqrt{19} \sqrt{3} x \\
        &+ 3158028 \, x y + 166212 \, \sqrt{19} \sqrt{\frac{2}{3}} z + 13851 \, \sqrt{19} \sqrt{3} \sqrt{2} \\
        &+ 44802 \, \sqrt{19} x - 169344 \sqrt{3} x - 97308 \, \sqrt{3} y \\
        &+ 249318 \, \sqrt{2} y - 43092 \, \sqrt{19} z - 264384 \, \sqrt{\frac{2}{3}} z \\
        &+ 27702 \, \sqrt{19} \sqrt{3} + 498636 \, y + 5814 \, \sqrt{19} - 44064 \, \sqrt{3}
    \end{aligned}
\end{align*}

It is a polynomial whose coefficients can be obtained independently using generalized formulas like \eqref{formulae_for_coefficients_ove_variable} and \eqref{formulae_for_coefficients} or also by direct substitution.
We verify that all coefficients are positive; hence, the polynomial contains no zeros at box $\mathcal{B}_{(0,\frac{1}{2}]}$.
Similar results hold for block $\mathcal{B}_{[\frac{1}{2}, 1)}$. This shows that $\gamma_1$ does not contain zeros at $(0,1)$ and completes the proof of Item {\it i)}.

Item {\it ii)}.
The proof is very similar to Item {\it i)}. Although no subdivision is needed in this case, the block corresponds to $\mathcal{B}_{(0,1)}$.

Item {\it iii)}. This item is an immediate consequence of items {\it i)} and  {\it ii)}. 

Item {\it iv)} For this item, a slight variation in 
the substitution was done to obtain a polynomial of a higher degree in $t$ and 
a smaller degree in $u_1$ and $u_2$. Equation \eqref{lift_for_the_cube}
implies that $(3t^2 + 2t + 3)^{\frac{3}{2}} = (3t^2 + 2t + 3)\cdot u_1$ and $(3t^2 - 2t + 3)^{\frac{3}{2}} = (3t^2 - 2t + 3)\cdot u_2$.
The block $\mathcal{C}_1$
has entries expressed as
\begin{align*}
(\mathcal{C}_1)_{11} &= 24 \, c - \frac{1}{12} \, \sqrt{3} - \frac{3}{8} \, \sqrt{2} - \frac{3}{4},\\
(\mathcal{C}_1)_{12} &= 24 \, c - \frac{\sqrt{3} t + \sqrt{3}}{3 \, {\left(t+ 1\right)^{3}}} - \frac{3 \, {\left(t + 3\right)}}{{\left(3 \, t^{2} + 2 \, t + 3\right)} u_1} + \frac{3 \, {\left(t - 3\right)}}{{\left(3 \, t^{2} - 2 \, t + 3\right)} u_2} - \frac{\sqrt{3} t - \sqrt{3}}{3 \, {\left(t - 1\right)}^{3}}, \\
(\mathcal{C}_1)_{21} &= 24 \, c t - \frac{\sqrt{3} t + \sqrt{3}}{3 \, {\left(t+ 1\right)^{3} }} - \frac{3 \, {\left(3 \, t + 1\right)}}{{\left(3 \, t^{2} + 2 \, t + 3\right)} u_1} - \frac{3 \, {\left(3 \, t - 1\right)}}{{\left(3 \, t^{2} - 2 \, t + 3\right)} u_2} + \frac{\sqrt{3} t - \sqrt{3}}{3 \, {\left(t - 1\right)}^{3}},\\
(\mathcal{C}_1)_{22} &= 24 \, c t - \frac{2 \, \sqrt{3} + 9 \, \sqrt{2} + 18}{24 \, t^{2}}.
\end{align*}

We calculate $\det(\mathcal{C}_1) = \gamma_1 c+\gamma_0$ and define $c(t, u_1, u_2):= -\frac{\gamma_0}{\gamma_1}$. We substitute this value of $c$ into $(\mathcal{C}_1)_{11}$ and take the numerator, resulting in the polynomial
\begin{align*}
\displaystyle -2916 \, \sqrt{3} \sqrt{2} t^{19} u_1^{2} u_2^{2} - 5832 \, \sqrt{3} t^{19} u_1^{2} u_2^{2} - 26244 \, \sqrt{2} t^{19} u_1^{2} u_2^{2} - 40338 \, t^{19} u_1^{2} u_2^{2}+\cdots \\
\cdots - 34992 \, t^{2} u_1^{2} u_2 
- 279936 \, t^{3} u_2^{2} + 34992 \, t^{2} u_1 u_2^{2} + 139968 \, t^{2} u_1^{2} - 139968 \, t^{2} u_2^{2},
\end{align*}
with degree $19, 2, 2$ into $t, u_1,$ and $u_2$ respectively.
Then, using partitions as before, we prove that the polynomial has no zeros at the covering $\mathcal{B}_{(0,\frac{1}{2}]}\cup \mathcal{B}_{[\frac{1}{2}, 1)} $. Additionally, we calculate the limits  
$\lim_{t \to 0^+}(\mathcal{C}_1)_{11}(c(t),t) = 0$ and 
$\lim_{t \to 1^-}(\mathcal{C}_1)_{11}(c(t),t) = +\infty$. The item {\it iv)} was proved.

Item {\it v).}
To prove this item, we prove the following three facts.
\begin{itemize}
  \item[a)] $(\mathcal{C}_1)_{12}(c(t),t) < 0$ for $t$ in  $[0, 1/2]$.
  \item[b)] The derivative of $(C_1)_{12}(c(t),t) > 0$ for $t$ in  $[1/2, 1]$.
  \item[c)]  $\lim_{t \to 1^-}(\mathcal{C}_1)_{12}(c(t),t) = +\infty$.
 \end{itemize}
 The proofs for the facts {\it a)} and {\it b)}  are similar to the previous items. Therefore, we do not give details here. Item {\it c)} follows from a direct calculation.

\subsubsection{Block $\mathcal{C}_2$}
\label{subsubsec:Block_C2_Cube}
The block $\mathcal{C}_2$ is a $2\times 2$ with entries given by
\begin{align*}
 (\mathcal{C}_2)_{11} &= \frac{1}{12} \, \sqrt{3} - \frac{3}{8} \, \sqrt{2} + \frac{3}{4},\\
 (\mathcal{C}_2)_{12} &= \frac{\sqrt{3} t + \sqrt{3}}{3 \, {\left(t+1\right)^3}} - \frac{3 \, {\left(t + 3\right)}}{{\left(3 \, t^{2} + 2 \, t + 3\right)}^{\frac{3}{2}}} - \frac{3 \, {\left(t - 3\right)}}{{\left(3 \, t^{2} - 2 \, t + 3\right)}^{\frac{3}{2}}} - \frac{\sqrt{3} t - \sqrt{3}}{3 \, {\left(t - 1\right)}^{3}}, \\
(\mathcal{C}_2)_{21}&=\frac{\sqrt{3} t + \sqrt{3}}{3 \, {\left(t + 1\right)^3}} - \frac{3 \, {\left(3 \, t + 1\right)}}{{\left(3 \, t^{2} + 2 \, t + 3\right)}^{\frac{3}{2}}} + \frac{3 \, {\left(3 \, t - 1\right)}}{{\left(3 \, t^{2} - 2 \, t + 3\right)}^{\frac{3}{2}}} + \frac{\sqrt{3} t - \sqrt{3}}{3 \, {\left(t - 1\right)}^{3}}, \\
(\mathcal{C}_2)_{22}&=\frac{2 \, \sqrt{3} - 9 \, \sqrt{2} + 18}{24 \, t^{2}}.
\end{align*}
This block does not depend on $c$. We calculate the determinant $\det(\mathcal{C}_2)$, make the lift given by $g$ in \eqref{lift_for_the_cube}, consider it a rational function, and take the polynomial at the numerator. A calculation shows that there are no zeros on block $\mathcal{B}_{(0, 1)},$ which contains the image of $g$.

\subsubsection{The blocks $\mathfrakC_7$ and $ \mathfrakC_8$}
\label{subsubsec:Block_C_7_C_8_Cube}
The blocks $\mathfrakC_7$ and $\mathfrakC_8$ are blocks of size $4\times 2$. We choose a convenient minor $2\times 2$ to analyze in these cases.
The same substitutions and similar analysis show that the chosen minors do not have zeros on $(0,1)$. So we will omit the details, the interested reader can consult the notebooks. This argument finishes the proof for the nested cube.

The discussion of subsections \ref{subsubsec:Block_C2_Cube} and \ref{subsubsec:Block_C_7_C_8_Cube}
 proves the following theorem.

\begin{theorem}
Consider a central configuration where the positions are the vertices of two nested cubes. The masses in each cube are equal. 
\end{theorem}

\section{Conclusion}
\label{Sec: Conclusion}
We summarize the discussion of this section in the following theorems. More details can be obtained in Sections \ref{Applications: The inverse and direct problem for central configurations of two nested tetrahedron, two nested octahedron} and \ref{Applications: The inverse and direct problem for central configurations of two nested cubes}.
\begin{theorem}
    Consider positions at the vertices of two nested tetrahedra or two nested octahedra or nested cubes. The configuration is only central if the masses in each polyhedron are equal.
\end{theorem}

\begin{theorem}
\label{final_theorem}    Consider positions at the vertices of two nested tetrahedra or two nested octahedra or nested cubes. For each of the three  cases, there is an interval $(0,\delta)$ such that if the ratio between edges $t$ belongs to this interval, there is a correspondent ratio between the masses that makes the configuration central. The relation between the ratio of masses can be explicitly obtained as a function of $t$.
    There is an interval $(\delta,1)$ such that no central configuration is possible for $t$ in this interval.
\end{theorem}
 The reader must contemplate a qualitative aspect of this last theorem: While near the collision, $t=1$, is prohibitive to obtain a central configuration, exhibiting behavior similar to Shub's theorem(see \cite{SHUB}). Near the collision $t=0$, we can always obtain central configurations changing the masses,
exhibiting a behavior remembering
 the works of Xia and Moeckel(\cite{Moeckel_Clusters,Xia}) about clusters of central configurations.

In \cite{Kuo_Chen_Measures}, Chen, Hsu, and Pan give a general lower bound of $\sqrt[3]{2}$ for the relationship between the radii of two concentric layers in a central configuration (see Theorem 5.2 and comment before Section 6 in this reference).
In our work, the fraction $\frac{1}{t}$ represents the ratio between the outer and inner radii of the two nested platonic polyhedra. Theorem \ref{final_theorem} gives a lower bound of $\dfrac{1}{\delta}$ for this ratio. 
 For comparison, we found approximately the lower bounds of $1.88, 1.72,$ and $1.63$ for the cases of tetrahedra, octahedra, and cube, respectively. 
\section{Data availability}
\label{Data availability}
The repository \url{https://github.com/ldsufrpe/block-analysis} contains all the data and code generated during the current study.
The calculations were done using SageMath software (\cite{SAGE}).

\begin{acknowledgements}
We acknowledge the financial support by  Fundação de
Amparo a Ciência e Tecnologia de Pernambuco (Grant No. APQ-1109-1.03/21). The authors would like to thank Eduardo S. G. Leandro for the helpful comments. To Rodrigo Gondim by the suggestion to use the lift as well as the Department of Mathematics at {\it Universidade Federal Rural de Pernambuco -UFRPE} for their assistance.
Also, the Department of Mathematics at {\it Universidade Federal de Pernambuco-UFPE} we acknowledge for the reception for a postdoc of one of the authors during part of the preparation of this work.
\end{acknowledgements}


\begin{thebibliography}{}
%
%
\bibitem{saari} Saari, Donald G. ``The manifold structure for collision and for hyperbolic-parabolic orbits in the n-body problem." Journal of differential equations 55.3 (1984): 300-329.

\bibitem{SmaleIntegralManifolds} Smale, Steven. ``Topology and mechanics. II." Inventiones mathematicae 11.1 (1970): 45-64.

\bibitem{CabralIntegralManifolds} Cabral, Hildeberto Eulalio. ``On the integral manifolds of the N-body problem." Inventiones mathematicae 20.1 (1973): 59-72.

\bibitem{AlbouyIntegralManifolds} Albouy, Alain. ``Integral manifolds of the N-body problem." Invent. math 114 (1993): 463-488.

\bibitem{MeyerHallOffin}  Meyer, K. R.; Hall, G. R.; Offin, D. {\it Introduction To Hamiltonian Dynamical Systems and The N-Body Problem.} Second edition. Nova York: Springer. Applied Mathematical Sciences Vol. 90. (2009)


\bibitem{LlibreMoeckelSimo}  Llibre, Jaume, Richard Moeckel, and Carles Sim\'o.{\it  Central Configurations, Periodic Orbits, and Hamiltonian Systems }. Birkh{\"a}user, (2015)



\bibitem{AlbouyCabralSantos} Albouy, A. , Cabral, H. E. , Santos, A. A., {\it Some problems on the Classical n-body problem}. Celest. Mech. Dyn. Astr. , vol 113. . p. 369-375. (2012)

\bibitem{SmaleProblems} Smale, S. {\it Mathematical Problems for the Next Century}. Math. Intelligencer 20, No. 2, p. 7-15 (1998) 

\bibitem{Lagrange} Lagrange, Joseph-Louis. ``Essai sur le probleme des trois corps." Prix de l’académie royale des Sciences de paris 9 (1772): 292.

\bibitem{SaariOntheRole} Saari, Donald G. ``On the role and the properties of n-body central configurations.'' Celestial mechanics 21.1 (1980): 9-20.





\bibitem{Euler} Euler, Leonhard. ``De motu rectilineo trium corporum se mutuo attrahentium." Novi commentarii academiae scientiarum Petropolitanae (1767): 144-151.

\bibitem{Moulton} Moulton, Forest Ray. ``The straight line solutions of the problem of N bodies." The Annals of Mathematics 12.1 (1910): 1-17.

\bibitem{MoeckelAbouy} Albouy, Alain, and Richard Moeckel. ``The inverse problem for collinear central configurations." Celestial Mechanics and Dynamical Astronomy 77.2 (2000): 77-91.

\bibitem{Candice} Davis, Candice, et al. ``Inverse problem of central configurations in the collinear 5-body problem." Journal of Mathematical Physics 59.5 (2018): 052902.

\bibitem{SANTOS_22} Santos, Marcelo P. ``Symmetric Central Configurations and the Inverse Problem." Journal of Dynamics and Differential Equations (2022): 1-21..





\bibitem{XIA_ZHOU} Xia, Zhihong, and Tingjie Zhou. ``Applying the symmetry groups to study the n body problem." Journal of Differential Equations 310 (2022): 302-326.

\bibitem{LEANDRO} Leandro, Eduardo SG. ``Factorization of the Stability Polynomials of Ring Systems." arXiv preprint arXiv:1705.02701 (2017).

\bibitem{LEANDRO_2} Leandro, Eduardo SG. ``Structure and stability of the rhombus family of relative equilibria under general homogeneous forces." Journal of Dynamics and Differential Equations 31.2 (2019): 933-958.

\bibitem{Corbera_Llibre_2n} Corbera, Montserrat, and Jaume Llibre. ``Central configurations of nested regular polyhedra for the spatial 2n-body problem." Journal of Geometry and Physics 58.9 (2008): 1241-1252.

\bibitem{Corbera_Llibre_3n} Corbera, Montserrat, and Jaume Llibre. ``Central configurations of three nested regular polyhedra for the spatial 3n-body problem." Journal of Geometry and Physics 59.3 (2009): 321-339.

\bibitem{Corbera_Llibre_Pn} Corbera, Montserrat, and Jaume Llibre. ``On the existence of central configurations of p nested regular polyhedra." Celestial Mechanics and Dynamical Astronomy 106.2 (2010): 197-207.

\bibitem{ZHU} Zhu, Changrong. ``Central configurations of nested regular tetrahedrons." Journal of mathematical analysis and applications 312.1 (2005): 83-92.

\bibitem{Liu_Tao} Liu, Xuefei, and Chaohai Tao. ``The Existence and Uniqueness of Central Configurations for Nested Regular Octahedron." Southeast Asian Bulletin of Mathematics 29.4 (2005).

\bibitem{Cedo_Llibre} Cedó, Ferran, and Jaume Llibre. ``Symmetric central configurations of the spatial n-body problem.'' Journal of Geometry and Physics 6.3 (1989): 367-394.

\bibitem{Montaldi} Montaldi, James. ``Existence of symmetric central configurations." Celestial Mechanics and Dynamical Astronomy 122.4 (2015): 405-418.


\bibitem{Serre} Serre, Jean-Pierre. ``Linear representations of finite groups''. Vol. 42. New York: springer, 1977.

\bibitem{Stiefel_Fassler} Stiefel, E., and A. Fässler. ``Group theoretical methods and their applications''. Springer Science \& Business Media, 2012.

\bibitem{Leandro_Finiteness_Bifurcation} Leandro, E. S. (2003). ``Finiteness and bifurcations of some symmetrical classes of central configurations''. Archive for rational mechanics and analysis, 167, 147-177.

\bibitem{RAMIREZ_GASULL_Llibre} Alvarez-Ramírez, M., Gasull, A., \& Llibre, J. (2022). ``On the equilateral pentagonal central configurations''. Communications in Nonlinear Science and Numerical Simulation, 112, 106511.

\bibitem{Gasull_LAzaro_Torregrosa} Gasull, Armengol, J. Tomás Lázaro, and Joan Torregrosa. ``Rational parameterizations approach for solving equations in some dynamical systems problems." Qualitative theory of dynamical systems 18.2 (2019): 583-602.

\bibitem{Uspensky} Uspensky, J.V., 1948. Theory of equations. Tata McGraw-Hill Education.

\bibitem{VCA} Collins, G.E. and Akritas, A.G., 1976, August. Polynomial real root isolation using Descarte's rule of signs. In Proceedings of the third ACM symposium on Symbolic and algebraic computation (pp. 272-275).

\bibitem{JMT} Thomas, J.M., 1941. Sturm's theorem for multiple roots. National Mathematics Magazine, 15(8), pp.391-394.



\bibitem{Barros_Leandro} Barros, Jean F., and Eduardo SG Leandro. ``Localization of real algebraic hypersurfaces with applications to the enumeration of the classes of relative equilibria of a (5+ 1)-body problem.'' Journal of Mathematical Analysis and Applications 485.2 (2020): 123813.

\bibitem{Alesina_Galuzzi1998} Alesina, Alberto, and Massimo Galuzzi. "A new proof of Vincent's theorem." Enseignement Mathématique 44 (1998): 219-256.

\bibitem{GAP4} GAP.
  The GAP~Group, GAP-Groups- Algorithms, and Programming, 
  Version 4.12.1; 
  2022,
  \url{https://www.gap-system.org}.




\bibitem{SHUB} Shub, M. ``Appendix to Smale’s paper: Diagrams and relative equilibria in manifolds, Amsterdam, 1970." Lecture Notes in Math 197 (1970): 4-14.

\bibitem{Moeckel_Clusters} Moeckel, Richard. ``Relative equilibria with clusters of small masses." Journal of Dynamics and Differential Equations 9 (1997): 507-533.

\bibitem{Xia} Xia, Zhihong. ``Central configurations with many small masses." Journal of Differential Equations 91.1 (1991): 168-179.

\bibitem{Kuo_Chen_Measures} Chen, Kuo-Chang, Ku-Jung Hsu, and Bo-Yu Pan. ``A Theory of Central Measures for Celestial Mechanics." SIAM Journal on Applied Dynamical Systems 16.1 (2017): 204-225.

\bibitem{SAGE} SageMath, the Sage Mathematics Software System (Version 9.5). The Sage Developers (2024). http://
www.sagemath.org



\end{thebibliography}


\end{document}